\documentclass{amsart}

\usepackage{teg}

\title{Tensor extriangulated categories}
\date{\today}



\author[R.~Bennett-Tennenhaus]{Raphael Bennett-Tennenhaus}
\address{Raphael Bennett-Tennenhaus,
Faculty of Mathematics,
Bielefeld University,
PO Box 100 131,
33501 Bielefeld,
Germany,
ORCiD: 0000-0003-2946-4223}
\email{raphaelbennetttennenhaus@gmail.com}

\author[I.~Goodbody]{Isambard Goodbody}
\address{Isambard Goodbody,
School of Mathematics and Statistics,
University of Glasgow,
University Place,
Glasgow G12 8QQ,
United Kingdom}
\email{isambard.goodbody@glasgow.ac.uk}

\author[J.~C.~Letz]{Janina C. Letz}
\address{Janina~C.~Letz,
Faculty of Mathematics,
Bielefeld University,
PO Box 100 131,
33501 Bielefeld,
Germany,
ORCiD: 0000-0002-5497-8296}
\email{jletz@math.uni-bielefeld.de}

\author[A.~Shah]{Amit Shah}
\address{Amit Shah,
Department of Mathematics, 
Aarhus University, 
Ny Munkegade 118, 
8000 Aarhus, 
Denmark,
ORCiD: 0000-0002-6623-8228}
\email{a.shah1728@gmail.com}

\begin{document}

\begin{abstract}
A tensor extriangulated category is an extriangulated category with a symmetric monoidal structure that is compatible with the extriangulated structure. To this end we define a notion of a biextriangulated functor $\cat{A} \times \cat{B} \to \cat{C}$, with compatibility conditions between the components. We have two versions of compatibility conditions, the stronger depending on the higher extensions of the extriangulated categories. We give many examples of tensor extriangulated categories. Finally, we generalise Balmer's classification of thick tensor ideals to tensor extriangulated categories.
\end{abstract}

\maketitle

\setcounter{tocdepth}{1}
\tableofcontents

%%%%%%%%%%%%%%%%%%%%%%%%%%%%%%%%%%%%%%%%%%%%%%%%%%%%%%%%%%%
%%%%%%%%%%%%%%%%%%%%%%%%%%%%%%%%%%%%%%%%%%%%%%%%%%%%%%%%%%%
\section{Introduction}
\label{sec:introduction}

Monoidal categories play an important role in algebra, representation theory and various other areas. The module category of a commutative ring and of a cocommutative Hopf algebra over a field naturally possess a monoidal structure. These give rise to tensor triangulated categories and finite tensor categories, which have become their own areas of study. The symmetric monoidal structure is at the heart of tensor triangular geometry, allowing one to attach a topological space to a tensor triangulated category; see \cite{Balmer:2005}. Further, a monoidal category is the natural setting for the categorical notion of (co)algebra objects. 

The aim of this paper is to define a compatibility condition of a monoidal structure with an extriangulated structure. Extriangulated categories were introduced by Nakaoka and Palu \cite{Nakaoka/Palu:2019}, and provide a generalisation of both exact and triangulated categories. An extriangulated structure is given by a class of extriangles, generalising exact sequences and exact triangles, and extension functors, generalising $\operatorname{Ext}^1$ of exact categories and $\Hom{}{-}{\susp -}$ of triangulated categories. 

We introduce \emph{biextriangulated} functors in \cref{sec_biextr_functor} building on extriangulated functors introduced in  \cite{BennettTennenhaus/Shah:2021}. Roughly speaking, a biextriangulated functor is a functor $\cat{A} \times \cat{B} \to \cat{C}$ that is extriangulated in each component, and where each associated natural transformation between extension functors is compatible with the other component. For many purposes this definition is adequate, though it does not recover bitriangulated functors, which involve $\Hom{}{-}{\susp^{2} -}$. To recover the notion of a bitriangulated functor we introduce the notion of a \emph{strong} biextriangulated functor for which we use \emph{higher extensions}. 

For exact and triangulated categories the higher extensions are $\Ext[n]{}{-}{-}$ and $\Hom{}{-}{\susp^{n}-}$. In \cref{sec-higer-exts} we recall the construction of higher extension functors: in the presence of enough injectives or projectives following \cite{Herschend/Liu/Nakaoka:2022}; and  using coends following \cite{Gorsky/Nakaoka/Palu:2021}. When any two of these constructions exist they are isomorphic. The higher extensions come equipped with a cup product. We show that whenever higher extensions exist, the corresponding cup product is compatible with  any extriangulated functor, natural transformation or adjunction. 

Finally, in \cref{sec-tensor-extriangulated} we define a category to be \emph{monoidal extriangulated} if it is monoidal, extriangulated, the tensor product is biextriangulated and the associator and unitors are extriangulated natural transformations. If the monoidal structure is symmetric we say the category is \emph{tensor extriangulated}. The higher extensions of the unit, together with the cup product, are the homogeneous components of a graded ring. This ring generalises the \emph{Ext algebra} of finite tensor  categories as well as the \emph{graded central ring} (of the unit) from tensor triangular  geometry; see for example \cite{Etingof-Ostrik:2004} and \cite{Balmer-spectra:2010}. In \cref{extension_ring_commutative} we see this is a particularly nice ring when the monoidal structure is a strong biextriangulated category:
\begin{introlem}
    If $(\cat{A},\BE,\fs, \otimes, \unit)$ is a strong monoidal extriangulated category then the cup product makes $\BE^*(\unit,\unit)=\{\BE^{n}(\unit,\unit)\}_{n \geqslant 0}$ into a graded-commutative graded ring. 
\end{introlem}

At the end of \cref{sec-biextria} and in \cref{sec:example_tec} we present various examples of (strong) biextriangulated functors and (strong) monoidal extriangulated categories:
\begin{itemize}
\item exact category of (bi)additive functors; see \cref{example-flat-functors,example_flat_bifunctors};
\item (stabilisation of the) category of matrix factorisations; see \cref{example-matrix-factorisation};
\item category of flat modules over a commutative ring; see \cref{example-local-units-flat};
\item (stabilisation of the) category of modules over a cocommutative Hopf algebra over a field; see \cref{example_Hopf}; and 
\item extriangulated structure induced by pure-exact triangles; see \cref{subsec-pure-exact}.
\end{itemize}
We also look into the behaviour of (strong) biextriangulated functors under stabilisation of extriangulated categories in \cref{subsec:stabilisation,subsec:stabilisation_ttc} and extriangulated substructures in \cref{sec:relative-structure-to-generate-new-examples-from-old}. 

In closing we generalise Balmer's classification \cite{Balmer:2005} of radical thick tensor ideals of a tensor triangulated category to tensor extriangulated categories. We use the approaches of Kock and Pitsch \cite{Kock/Pitsch:2017} and Buan, Krause and Solberg \cite{Buan/Krause/Solberg:2005}, and the proofs can be directly transferred to the extriangulated setting. 
\begin{introthm}
\label{introtheorem-full-radical-bijection}\renewcommand{\qedsymbol}{}
Let $(\cat{A},\BE,\fs,\otimes, \unit)$ be a tensor extriangulated category. We assume the class of radical thick tensor ideals $\Rad(\cat{A})$ is a set. Then there exists a spectral space $\balmer{\cat{A}}$ and a bijection
\begin{equation*}
\Rad(\cat{A}) \longleftrightarrow \left\{\text{Thomason closed subsets of } \balmer{\cat{A}}\right\}
\nospacepunct{.}
% \,. \pushQED{\qed} \qedhere \popQED
\end{equation*}
\end{introthm}
This result is contained in \cref{theorem-full-radical-bijection}. We compute the spectral space $\balmer{\cat{A}}$ for the category of projective modules over various commutative rings in \cref{teg_projective_modules} and for stabilisations of extriangulated categories in \cref{teg_stabilisation}. 

\begin{ack} 
R.~B-T.\@ is grateful to Isaac Bird for helpful discussions concerning \cref{subsec-pure-exact}. R.~B-T.\@ and A.~S. \@ were partly funded by the Danish National Research Foundation (DNRF156); the Independent Research Fund Denmark (1026-00050B); the Aarhus University Research Foundation (AUFF-F-2020-7-16). R.~B-T.\@ and J.~C.~L.\@ were partly funded by the Deutsche Forschungsgemeinschaft (DFG, German Research Foundation)---Project-ID 491392403---TRR 358. J.~C.~L.\@ was partly funded by the Alexander von Humboldt Foundation in the framework of a Feodor Lynen research fellowship endowed by the German Federal Ministry of Education and Research. I. G. was funded by a PhD scholarship from the Carnegie Institute for the Universities of Scotland. 
\end{ack}

%%%%%%%%%%%%%%%%%%%%%%%%%%%%%%%%%%%%%%%%%%%%%%%%%%%%%%%%%%%
%%%%%%%%%%%%%%%%%%%%%%%%%%%%%%%%%%%%%%%%%%%%%%%%%%%%%%%%%%%
\section{Higher extensions in extriangulated categories}
\label{sec-higer-exts}

The notion of extriangulated categories generalises exact categories and triangulated categories. 
They were introduced by Nakaoka and Palu \cite{Nakaoka/Palu:2019}. 
There exist notions of higher extensions in exact categories (under some mild conditions) and in triangulated categories. 
These have been unified for extriangulated categories in \cite{Liu/Nakaoka:2019,Herschend/Liu/Nakaoka:2022,Gorsky/Nakaoka/Palu:2021}. In this section we recall the definitions and constructions and discuss some properties.

%%%%%%%%%%%%%%%%%%%%%%%%%%%%%%%%%%%%%%%%%%%%%%%%%%%%%%%%%%%
\subsection{Extriangulated categories}

An extriangulated category $\cat{A} = (\cat{A},\BE,\fs)$ consists of an additive category $\cat{A}$, a functor $\BE \colon \op{\cat{A}} \times \cat{A} \to \abcat$, additive in each component, to the category $\abcat$ of abelian groups, and a correspondence $\fs $ associating  an equivalence class $\fs (d)$ of a 3-term complex $X \to W \to Y$ to each $X,Y \in \cat{A}$ and $d \in \BE(Y,X)$,  satisfying some axioms; see \cite[Definition 2.12]{Nakaoka/Palu:2019}. 
By an \emph{extriangle} we mean a representative of an equivalence class given by $\fs$, written 
\[
\begin{array}{cc}
X \xrightarrow{f} W \xrightarrow{g} Y \xdashedrightarrow{d}
&
\text{where }
\fs(d) = [X \xrightarrow{f} W \xrightarrow{g} Y]\,. 
\end{array}
\]
Let $R$ be a commutative ring. An extriangulated category $(\cat{A},\BE,\fs)$ is \emph{$R$-linear}, if the underlying category $\cat{A}$ is $R$-linear, the functor $\BE$ factors through the forgetful functor $\Mod{R} \to \abcat$ and the induced functor $\op{\cat{A}} \times \cat{A} \to \Mod{R}$ is $R$-bilinear; see \cite[Section~3.1]{Iyama/Nakaoka/Palu:2024}. We denote the functor $\op{\cat{A}} \times \cat{A} \to \Mod{R}$ also by $\BE$. Every extriangulated category is trivially $\BZ$-linear. 

Like exact and triangulated categories, to be an extriangulated category is to equip a category with additional structure; it is not a property of the category. Any exact or triangulated category has a natural extriangulated structure:

\begin{example} \label{example_exact}
Let $\cat{E}$ be a (Quillen-)exact category. This means $\cat{E}$ is equipped with a class of short exact sequences. Exact categories were introduced in \cite[Section~2]{Quillen:1973}, for more details see \cite{Buehler:2010}. When $\cat{E}$ is an essentially small category or $\cat{E}$ has enough $\cat{E}$-projectives or enough $\cat{E}$-injectives, then $\cat{E}$ is an extriangulated category with $\BE(Y,X) \coloneqq \Ext[1]{\cat{E}}{Y}{X}$, the collection of Baer equivalence classes $[X \to W \to Y]$, and $\fs$ the identity; for more details see \cite[Example~2.13]{Nakaoka/Palu:2019}.
% \begin{equation*}
%  = \set{| \text{equivalence classes of exact sequence}}
% \end{equation*}
 The assumptions on $\cat{E}$ are sufficient to ensure that $\BE(Y,X)$ is a set. 
\end{example}

\begin{example} \label{example_triangulated}
Let $\cat{T}$ be a triangulated category with suspension functor $\susp$. This means $\cat{T}$ is equipped with a class of exact triangles. Triangulated categories were introduced by Verdier \cite{Verdier:1996}. Then $\cat{T}$ is an extriangulated category with 
\begin{equation*}
\BE(Y,X) \coloneqq \Hom{\cat{T}}{Y}{\susp X} \quad \text{and} \quad \fs(d) \coloneqq [X \to \susp^{-1} \cone(d) \to Y]
\end{equation*}
for any $d \in \BE(Y,X)$; for more details see \cite[Proposition~3.22(1)]{Nakaoka/Palu:2019}. 
\end{example}

Other examples of extriangulated categories include extension-closed full subcategories of triangulated categories \cite[Remark~2.18]{Nakaoka/Palu:2019}, and certain ideal quotients of exact categories \cite[Proposition~3.30]{Nakaoka/Palu:2019}. We discuss the latter in more detail in \cref{subsec:stabilisation}. 

In \Cref{example_triangulated,example_exact} a notion of higher extensions exists. Motivated by higher extensions for exact categories, there are two constructions for higher extensions in extriangulated categories: via coends by \cite{Gorsky/Nakaoka/Palu:2021} when the underlying category is essentially small, and via syzygies or cosyzygies by \cite{Liu/Nakaoka:2019,Herschend/Liu/Nakaoka:2022} when the extriangulated category has enough projectives or injectives, respectively. In the following sections we recall both constructions.

%%%%%%%%%%%%%%%%%%%%%%%%%%%%%%%%%%%%%%%%%%%%%%%%%%%%%%%%%%%
\subsection{Higher extensions via coends} \label{higher_ext_coend}

We briefly recall the definition of a coend. Let $\mathsf{F} \colon \op{\cat{A}} \times \cat{A} \to \cat{B}$ be a functor where $\cat{A}$ is essentially small and $\cat{B}$ has set-indexed colimits. A \emph{coend} of $\mathsf{F}$ is the coequaliser of the diagram
\begin{equation*}
\begin{tikzcd}
\coprod\limits_{V,V' \in \sk(\cat{A})} \coprod\limits_{f \in \Hom{\cat{A}}{V}{V'}} \mathsf{F}(V',V) \ar[r,shift left] \ar[r,shift right] & \coprod\limits_{W \in \sk(\cat{A})} \mathsf{F}(W,W) \nospacepunct{,}
\end{tikzcd}
\end{equation*}
defined by the morphisms $\mathsf{F}(f,V)$ and $\mathsf{F}(V',f)$ where  $\sk(\cat{A})$ denotes the skeleton of $\cat{A}$, ensuring the coproducts are set-indexed.
%; cf.\@ \cite[Definiton~1.5.15]{Riehl:2016}. 
In particular a coend is a colimit and, if it exists, is unique up to unique isomorphism. It is denoted by $\int^W \mathsf{F}(W,W)$, and is equipped with morphisms $\mathsf{F}(X,X) \to \int^W \mathsf{F}(W,W)$ for every $X \in \cat{A}$.

Let $R$ be a commutative ring and let $(\cat{A},\BE,\fs)$ be an essentially small $R$-linear extriangulated category. The \emph{higher extensions} are functors
\begin{equation} \label{defn_higher_ext_coend}
\begin{gathered}
\BE^n \colon \op{\cat{A}} \times \cat{A} \to \Mod{R} \quad \text{given by} \\
\BE^0(Y,X) \coloneqq \Hom{\cat{A}}{Y}{X} \quad \text{and} \quad \BE^n(Y,X) \coloneqq \int^W \BE^{n-1}(W,X) \otimes_R \BE(Y,W)
\end{gathered}
\end{equation}
for $n \geq 1$; that is, the object $\BE^n(Y,X)$ is the coend of the functor $\BE^{n-1}(-,X) \otimes_R \BE(Y,-)$. By \cite[Proposition~2.2.1]{Loregian:2021}, we have $\BE^1 = \BE$. Using the universality of the colimit, it is straightforward to check that $\BE^n(Y,X)$ is functorial in $X$ and $Y$, and using also that tensoring over $R$ is right exact and commutes with colimits, one has natural isomorphisms
\begin{equation}
\label{cup-prod-coend-prelims}
\int^W \BE^i(W,X) \otimes_R \BE^j(Z,W) \xrightarrow{\cong}\BE^n(Z,X)  
\end{equation}
for any integers $i,j\geq 0$ with $i+j = n$. This immediately yields the \emph{cup product}
\begin{equation} \label{cup_product_coend}
\smile \colon \BE^i(Y,X) \otimes_R \BE^j(Z,Y) \to \BE^n(Z,X)
\end{equation}
given by composing  the morphism in  \eqref{cup-prod-coend-prelims} with the canonical morphism 
\[
\BE^{i}(Y,X)\otimes_{R}\BE^{j}(Z,Y)\to \int^W \BE^i(W,X) \otimes_R \BE^j(Z,W)\,. 
\]
This is motivated by \cite[Definition~3.10]{Gorsky/Nakaoka/Palu:2021} where the cup product is introduced for $i+j=1$. We have 
\begin{equation} \label{higher_ext_coend_Hom}
% \int^W \Hom{\cat{A}}{W}{X} \otimes_R \BE^n(Y,W) = \BE^n(Y,X) \quad \text{and} \quad f \smile d = \BE^n(Z,f)(d)
f \smile d = \BE^n(Z,f)(d) \quad \text{and} \quad e \smile g = \BE^n(g,X)(e)
\end{equation}
for morphisms $f \colon Y \to X$ and $g \colon Z \to Y$ in $\cat{A}$ and $d \in \BE^n(Z,Y)$ and $e \in \BE^n(Y,X)$. 

This construction of higher extension groups for exact categories goes back to \cite[Section~4]{Yoneda:1960}. For extriangulated categories this was foreshadowed in \cite[Remark~5.10]{Nakaoka/Palu:2019} for the case where $i=j=1$ in \cref{cup_product_coend}. 
It then was extended to any $i,j\geq 0$ in \cite{Gorsky/Nakaoka/Palu:2021}. 

\begin{remark}
The construction in \cref{defn_higher_ext_coend} can be done more generally to obtain a ``folding'' of $R$-bilinear functors. This defines a tensor product of $R$-bilinear functors. The tensor product already appears in \cite[Section~4]{Yoneda:1960};  for a short survey, see \cite{ElKaoutit:2020}. We discuss this later in \cref{example-flat-functors}. 
\end{remark}

\begin{example} \label{example_exact_coend}
Let $\cat{E}$ be an exact category that is essentially small. We consider $\cat{E}$ as an extriangulated category, as in \cref{example_exact}. Roughly speaking, the coend acts as gluing exact sequences. This means $\BE^n(Y,X)$ is the set of equivalence classes
\begin{equation*}
% [X = X_0 \xrightarrow{f_0} X_1 \xrightarrow{f_1}  \cdots  \xrightarrow{f_n} X_{n+1} = Y],
[X = X_{0}\to  X_{1} \to  \cdots \to  X_{n+1} = Y]\,,
\end{equation*}
where $X_{1},\dots,X_{n}\in \cat{E}$  and there is a commutative diagram of the form
\[
% \begin{tikzcd}
%     	X^{0} \arrow{rr}{d_{X}^{0}}\arrow[swap, equals]{dr}{g^{0}}&&
%         X^{1}\arrow[swap]{dr}{g^{1}} \arrow{rr}{d_{X}^{1}}&&
%        \cdots \arrow{rr}{d_{X}^{2}}\arrow[swap]{dr}{g^{n}}&&
%         X^{n+1} 
%         \\
%         {} &
%         W^{0}\arrow[swap]{ur}{f^{0}}
%         &&W^{1}\arrow[swap]{ur}{f^{1}}&&
%         W^{n}\arrow[swap,equals]{ur}{f^{n}}
%         &
% \end{tikzcd}
\begin{tikzcd}[column sep = 0.3cm]
    	X^{0} \arrow{rr}\arrow[swap, equals]{dr}&&
        X^{1}\arrow[swap]{dr} \arrow{rr}&&
       \cdots \arrow{rr}\arrow[swap]{dr}&&
        X^{n}\arrow{rr}\arrow[swap]{dr}&&
        X^{n+1} 
        \\
        {} &
        W^{0}\arrow[swap]{ur}
        &&W^{1}\arrow[swap]{ur}&&
        W^{n-1}\arrow[swap]{ur}
        &&W^{n}\arrow[swap,equals]{ur}&
\end{tikzcd}
\]
such that 
$W_{i} \to  X_{i+1} \to W_{i+1}$ is exact for each  $i=0,\dots,n-1$; see \cite[Section~3.4]{Yoneda:1954} and \cite[Section~4.3]{Yoneda:1960}.
% for each  $i=0,\dots,n-1$  the morphism $f_{i}$ factors as $X_{i} \xrightarrow{h_{i}}W_{i} \xrightarrow{g_{i+1}} X_{i+1}$ such that  $W_{i} \xrightarrow{g_{i+1}}  X_{i+1} \xrightarrow{h_{i+1}} W_{i+1}$ is exact and such that $W_{0} = X$, $W_{n} = Y$, $h_{0}=\id_{X}$ and $g_{n+1}=\id_{Y}$
 The cup product is given by
\begin{equation*}
\begin{aligned}
[X = X_0 \to \cdots \to X_{n+1} & = Y] \smile [Y = Y_0 \to \cdots \to Y_{m+1} = Z] \\
&= [X = X_0 \to \cdots \to X_n \to Y_1 \to \cdots \to Y_{m+1} = Z]
\end{aligned}
\end{equation*}
where the morphism $X_n \to Y_1$ is the composition $X_n \to Y \to Y_1$.
\end{example}

\begin{example} \label{example_triangulated_coend}
Let $\cat{T}$ be a triangulated category with suspension functor $\susp$. We consider $\cat{T}$ as an extriangulated category, as in \cref{example_triangulated}. By \cite[Corollary~3.23]{Gorsky/Nakaoka/Palu:2021} we have
\begin{equation*}
\BE^n(Y,X) = \Hom{\cat{T}}{Y}{\susp^n X}\,,
\end{equation*}
and for $d \in \BE^i(Y,X)$ and $e \in \BE^j(Z,Y)$ the cup product is
\begin{equation*}
d \smile e = (\susp^j d) \circ e \in \Hom{\cat{T}}{Z}{\susp^{i+j} X} = \BE^{i+j}(Z,X)\,.
\end{equation*}
\end{example}

%%%%%%%%%%%%%%%%%%%%%%%%%%%%%%%%%%%%%%%%%%%%%%%%%%%%%%%%%%%
\subsection{Higher extensions via syzygies or cosyzygies} \label{higher_ext_proj_inj}

We recall from \cite[Section~3.4]{Nakaoka/Palu:2019} the definition of an extriangulated category with enough projectives. Let $(\cat{A},\BE,\fs)$ be an extriangulated category. An object $P \in \cat{A}$ is \emph{$\BE$-projective}, if $\BE(P,X) = 0$ for any $X \in \cat{A}$. We say $(\cat{A},\BE,\fs)$ has \emph{enough $\BE$-projectives} if for any $X \in \cat{A}$ there exists an $\BE$-projective object $P$ and an extriangle $W \to P \to X \dashrightarrow$. The object $W$ is called a \emph{syzygy} of $X$. The syzygy of an object $X$ need not be unique. When $\cat{A}$ has enough projectives, we fix a choice of syzygy, and denote it by $\Omega X$; cf.\@ \cite[Assumption~3.3]{Herschend/Liu/Nakaoka:2022}. We set
\begin{equation*}
\Omega^0 X \coloneqq X \quad \text{and} \quad \Omega^n X \coloneqq \Omega (\Omega^{n-1} X)
\end{equation*}
for $n \geq 1$. 

Let $R$ be a commutative ring. In an $R$-linear extriangulated category $(\cat{A},\BE,\fs)$ with enough projectives, the \emph{higher extensions} are functors
\begin{equation} \label{higher_ext_syz}
\begin{gathered}
\BE^{n}_{\Omega} \colon \op{\cat{A}} \times \cat{A} \to \Mod{R} \quad \text{given by} \\
\BE^{0}_{\Omega}(Y,X) \coloneqq \Hom{\cat{A}}{Y}{X} \quad \text{and} \quad \BE^{n}_{\Omega}(Y,X) \coloneqq \BE(\Omega^{n-1} Y,X)
\end{gathered}
\end{equation}
for $n \geq 1$; these functors are well-defined and independent of the choice of syzygy by \cite[Proposition~3.4]{Herschend/Liu/Nakaoka:2022}. 

Dually, in an extriangulated category $(\cat{A},\BE,\fs)$ an object $I \in \cat{A}$ is \emph{$\BE$-injective}, if $\BE(Y,I) = 0$ for any $Y \in \cat{A}$. We say $(\cat{A},\BE,\fs)$ has \emph{enough $\BE$-injectives} if for any $X \in \cat{A}$ there exists an $\BE$-injective object $I$ and an extriangle $X \to I \to W \dashrightarrow$. The object $W$ is a \emph{cosyzygy} of $X$. In an $R$-linear extriangulated category $(\cat{A},\BE,\fs)$ with enough injectives we fix a choice of coszyzy $\susp X$ for every $X \in \cat{T}$. The \emph{higher extensions} are the functors
\begin{equation} \label{higher_ext_cosyz}
\begin{gathered}
\BE^{n}_{\susp} \colon \op{\cat{A}} \times \cat{A} \to \Mod{R} \quad \text{given by} \\
\BE^{0}_{\susp}(Y,X) \coloneqq \Hom{\cat{A}}{Y}{X} \quad \text{and} \quad \BE^{n}_{\susp}(Y,X) \coloneqq \BE(Y,\susp^{n-1} X)
\end{gathered}
\end{equation}
for $n \geq 1$; these functors are well-defined and independent of the choice of cosyzygy by \cite[Proposition~3.4]{Herschend/Liu/Nakaoka:2022}.

When the extriangulated category $(\cat{A},\BE,\fs)$ has enough $\BE$-injectives and enough $\BE$-projectives, the definitions of higher extensions via syzygies and cosyzygies are naturally isomorphic by \cite[Lemma~5.1]{Liu/Nakaoka:2019}.

\begin{example}
Let $\cat{T}$ be a triangulated category with suspension $\susp$. We consider $\cat{T}$ as an extriangulated category as in \cref{example_triangulated}. The zero object is the only $\BE$-projective object, and for any object $X$ there exists an exact triangle $\susp^{-1} X \to 0 \to X \xrightarrow{\id_X} X$. Hence $\cat{T}$ has enough $\BE$-projectives and $\susp^{-1} X$ is a syzygy of $X$. Dually, any triangulated category has enough injectives and the suspension $\susp X$ is a cosyzygy of $X$. It is straightforward to see that the higher extensions of any triangulated category $\cat{T}$ are given by 
\begin{equation*}
\BE_\susp^n(Y,X) = \Hom{\cat{T}}{Y}{\susp^n X} \cong \Hom{\cat{T}}{\susp^{-n} Y}{X} = \BE_\Omega^n(Y,X)\,.
\end{equation*}
\end{example}

Next we discuss the cup product on $\BE_\Omega$. The construction for $\BE_\susp$ works analogously. Let $R$ be a commutative ring and let $(\cat{A},\BE,\fs)$ be an $R$-linear extriangulated category with enough $\BE$-projectives. For an object $W\in\cat{A}$ there exists an $\BE$-projective object $P$ and an extriangle $\Omega W \to P \to W \dashrightarrow$. For any object $Y$ this extriangle induces a surjective map of $R$-modules
\begin{equation*}
% \Hom{\cat{A}}{Z}{Y} \to \Hom{\cat{A}}{P}{Y} \to 
\Hom{\cat{A}}{\Omega W}{Y} \to \BE(W,Y) \to 0
\end{equation*}
which is natural in $Y$; see \cite[Corollary~3.12]{Nakaoka/Palu:2019}. Then the cup product for $\BE_\Omega$ is defined as
\begin{equation} \label{cup_product_cosyz}
\begin{gathered}
\BE_\Omega^i(Y,X) \otimes_R \BE_\Omega^j(Z,Y) \to \BE_\Omega^i(\Omega^j Z,X) \cong \BE_\Omega^{i+j}(Z,X)\,, \\
d\otimes e\mapsto d \smile e \coloneqq \BE_\Omega^i(f_{e},X)(d)
\end{gathered}
\end{equation}
% for $d \in \BE_{\Omega}^{i}(Y,X)$, $e \in \BE_{\Omega}^{j}(Z,Y)$ and
where $f_{e}$ is a pre-image of $e$ under the surjective map above for $W = \Omega^{j-1} Z$.

%lies  
% \begin{equation*}
% \Hom{\cat{A}}{\Omega^j Z}{Y} \to \BE(\Omega^{j-1} Z,Y) = \BE_\Omega^j(Z,Y)\,.
% \end{equation*}
By \cite[Section~3.1]{Herschend/Liu/Nakaoka:2022} the cup product is independent of the pre-image, and there are natural isomorphisms $\BE_\Omega^k(\Omega^\ell -,-) \cong \BE_\Omega^{k+\ell}(-,-)$.

%%%%%%%%%%%%%%%%%%%%%%%%%%%%%%%%%%%%%%%%%%%%%%%%%%%%%%%%%%%
\subsection{Comparison} \label{higher_ext_comparision}

When $(\cat{A},\BE,\fs)$ is essentially small and has enough $\BE$-injectives or $\BE$-projectives, then the definitions of higher extensions are naturally isomorphic by \cite[Corollary~3.21]{Gorsky/Nakaoka/Palu:2021}. 
The proof in \textit{loc.cit.}\@ involves proving that the higher extensions and cup products are universal among bifunctors satisfying similar properties; see \cite[Proposition~3.20]{Gorsky/Nakaoka/Palu:2021}. 
We give an alternative proof, as well as a proof that the cup products are compatible with these natural isomorphisms. The main argument for both statements is contained in the following technical lemma.
%, and uses that colimits, and hence coends, are right exact.

\begin{lemma} \label{higher_ext_coend_cosyz_iso}
Let $R$ be a commutative ring and $(\cat{A},\BE,\fs)$ be an essentially small $R$-linear extriangulated category with enough $\BE$-projectives. There exist $R$-module homomorphisms $\phi_{i,j}(Y,X) \colon \BE^i(\Omega^j Y,X) \to \BE^{i+j}(Y,X)$ for $i,j \geq 0$ and objects $X,Y \in \cat{A}$, such that:
\begin{enumerate}
\item \label{higher_ext_coend_cosyz_iso:surj} $\phi_{i,j}(Y,X)$ is surjective when $i = 0$ and $j > 0$;
\item \label{higher_ext_coend_cosyz_iso:iso} $\phi_{i,j}(Y,X)$ is an isomorphism when $i \geq 1$ or $j = 0$;
\item \label{higher_ext_coend_cosyz_iso:naturalX} $\phi_{i,j}(Y,X)$ is natural in $X$; 
\item \label{higher_ext_coend_cosyz_iso:naturalY} $\phi_{i,j}(Y,X)$ is natural in $Y$ in the sense that the associated diagram commutes for any lift $\Omega^j Y \to \Omega^j Y'$ of a morphism $Y \to Y'$; and
\item \label{higher_ext_coend_cosyz_iso:cup} the morphisms are compatible with the cup product \cref{cup_product_coend} in the sense that the following diagram commutes
\begin{equation*}
\begin{tikzcd}
\BE^i(Y,X) \otimes_R \BE^j(\Omega^k Z,Y) \ar[r,"\smile"] \ar[d,"{\id \otimes \phi_{j,k}(Z,Y)}" swap] & \BE^{i+j}(\Omega^k Z,X) \ar[d,"{\phi_{i+j,k}(Z,X)}"] \\
\BE^i(Y,X) \otimes_R \BE^{j+k}(Z,Y) \ar[r,"\smile"] & \BE^{i+j+k}(Z,X)
\end{tikzcd}
\end{equation*}
for all integers $i,j,k \geq 0$. 
\end{enumerate}
\end{lemma}
\begin{proof}
We use induction on $j$. For $j=0$ we let $\phi_{i,0}(Y,X)$ be the identity morphism. For $j=1$ we have an extriangle $\Omega Y \to P \to Y \dashrightarrow$ with $P$ an $\BE$-projective object. Then there is a right exact sequence
\begin{equation*}
\Hom{\cat{A}}{P}{W} \to \Hom{\cat{A}}{\Omega Y}{W} \to \BE(Y,W) \to 0
\end{equation*}
in $\Mod{R}$ that is natural in $W \in \cat{A}$. We apply $\int^W \BE^i(W,X) \otimes_R -$ to the right exact sequence. Since coends preserve right exact sequences,  the sequence
\begin{equation*}
\BE^i(P,X) \to \BE^i(\Omega Y,X) \to \BE^{i+1}(Y,X) \to 0
% \int^W \BE(W,Y) \otimes \Hom{\cat{A}}{P_{n-1}}{W} \to \int^W \BE(W,Y) \otimes \Hom{\cat{A}}{\Omega^{n-1} X}{W} \to \int^W \BE(W,Y) \otimes \BE(\Omega^{n-2} X,W) \to 0
\end{equation*}
is exact; see \cite[Proposition~2.2.1]{Loregian:2021}. We let $\phi_{i,1}(Y,X)$ be the latter morphism. Finally, for $j > 1$ we define $\phi_{i,j}(Y,X)$ as the composition
\begin{equation*}
\BE^i(\Omega^j Y,X) \xrightarrow{\phi_{i,1}(\Omega^{j-1} Y,X)} \BE^{i+1}(\Omega^{j-1} Y,X) \xrightarrow{\phi_{i+1,j-1}(Y,X)} \BE^{i+j}(Y,X)\,.
\end{equation*}
The properties \cref{higher_ext_coend_cosyz_iso:surj,higher_ext_coend_cosyz_iso:iso,higher_ext_coend_cosyz_iso:naturalX} hold by construction. Property \cref{higher_ext_coend_cosyz_iso:naturalY} is straightforward to check from the construction, and \cref{higher_ext_coend_cosyz_iso:cup} can be shown using induction on $k$. 
\end{proof}

\begin{corollary}
    \label{cor:higher_ext_equiv}
    Let $R$ be a commutative ring and $(\cat{A},\BE,\fs)$ be an essentially small $R$-linear extriangulated category with enough $\BE$-projectives. 
    There are natural isomorphisms $\BE_{\Omega}^{i} \Longrightarrow  \BE^{i}$ for $i\geq 0$ that are compatible with the cup products  defined in \cref{cup_product_coend} and \cref{cup_product_cosyz}; that is the following diagram commutes
    \[
    \begin{tikzcd}
        \BE_{\Omega}^{i}(Y,X)\otimes \BE_{\Omega}^{j}(Z,Y)
        \arrow[d, "{\cong}"]
        \arrow[rrr, "{\smile}", "{\cref{cup_product_cosyz}}"']
        &
        &
        &
        \BE_{\Omega}^{i+j}(Z,X)\arrow[d, "{\cong}"]
        \\
        \BE^{i}(Y,X)\otimes \BE^{j}(Z,Y)
        \arrow[rrr, "{\smile}", "{\cref{cup_product_coend}}"']
        &
        &
        &
        \BE^{i+j}(Z,X) \nospacepunct{.}
    \end{tikzcd}
    \]
\end{corollary}

\begin{proof}
    For $i=0$ take the identity on $\Hom{\cat{A}}{Y}{X}$. 
    For $i>0$ we use the maps $\varphi_{1,i-1}(Y,X)\colon \BE_{\Omega}^{i}(Y,X)=\BE(\Omega^{i-1} Y,X)\to  \BE^{i}(Y,X)$ from \Cref{higher_ext_coend_cosyz_iso} 
    which are isomorphisms by \cref{higher_ext_coend_cosyz_iso:iso}. 
    We prove the square commutes by induction on $j$. 

    Let $d\in \BE_{\Omega}^{i}(Y,X)$ and $e\in \BE_{\Omega}^{j}(Z,Y)$. 
    For the case $j=0$ we obtain
    \begin{equation*}
    \begin{aligned}
\phi_{1,i-1}(Z,X)(d \smile e) &= \phi_{1,i-1}(Z,X)(\BE_\Omega^i(e,X)(d)) \\
&= \phi_{1,i-1}(Z,X)(\BE(\Omega^{i-1} e,X)(d)) \\
&= \BE^i(e,X)(\phi_{1,i-1}(Y,X)(d)) \\
&= \phi_{1,i-1}(Y,X)(d) \smile e\,;
    \end{aligned}
    \end{equation*}
    using property \cref{higher_ext_coend_cosyz_iso:naturalY} for the third identification. Note that $\Omega^{i-1} e$ is not a well-defined morphism, however the above identifications hold for any choice.

    We assume that the desired square commutes for some $j \geq 0$. We now claim that the following diagram commutes:
\[
\begin{tikzcd}
    \BE_{\Omega}^{i}(Y,X)\otimes \BE_{\Omega}^{j+1}(Z,Y)
        \arrow[d]
        \arrow[rr, "{\smile}", "{\cref{cup_product_cosyz}}"']
        &
        &
        \BE_{\Omega}^{i+j+1}(Z,X)\arrow[d]
    \\
    \BE_{\Omega}^{i}(Y,X)\otimes \BE_{\Omega}^{j}(\Omega Z,Y)
        \arrow[d,"{\phi_{1,i-1}(Y,X) \otimes \phi_{1,j-1}(\Omega Z,Y)}", swap]
        \arrow[rr, "{\smile}", "{\cref{cup_product_cosyz}}"']
        &&
        \BE_{\Omega}^{i+j}(\Omega Z,X)\arrow[d, "{\phi_{1,i+j-1}(\Omega Z,X)}"]
    \\
            \BE^{i}(Y,X) \otimes_R \BE^{j}(\Omega Z,Y) \ar[rr,"\smile", "{\cref{cup_product_coend}}"'] \ar[d,"{\id_{\BE^{i}(Y,X)} \otimes \phi_{j,1}(Z,Y)}", swap]
    &&
    \BE^{i+j}(\Omega Z,X) \ar[d,"{\phi_{i+j,1}(Z,X)}"]
    \\
    \BE^{i}(Y,X) \otimes_R \BE^{j+1}(Z,Y) \ar[rr,"\smile", "{\cref{cup_product_coend}}"'] 
    &&
    \BE^{i+j+1}(Z,X) \nospacepunct{.}
\end{tikzcd}
\]
The upper square commutes by the definition of $\BE_\Omega$ and its cup product. The middle square commutes by replacing $Z$ with $\Omega Z$ in the inductive hypothesis. 
The lower square commutes by \cref{higher_ext_coend_cosyz_iso:cup}. 
It remains to observe that the outside square is precisely the claim. 
\end{proof}

\begin{remark} \label{higher_ext_ring_bimod}
Let $(\cat{A},\BE,\fs)$ be an extriangulated category. Assume that higher extensions exist. Then $\BE^*(X,X) = \{\BE^n(X,X)\}_{n \geqslant 0}$ together with the cup product is a graded ring for any $X \in \cat{A}$. Moreover, for any $X,Y \in \cat{A}$ the cup product induces a $\BE^*(X,X)$-$\BE^*(Y,Y)$-bimodule structure on $\BE^*(Y,X) = \{\BE^n(Y,X)\}_{n \geqslant 0}$. 

Note that there is a category where the objects are the objects of $\cat{A}$, where $\BE^*(X,Y)$ is the set of morphisms $X \to Y$,  and where the cup product defines composition. This is a special case of the \emph{additive tensor category} in \cite{simson-cats-of-reps-of-species}.
\end{remark}

\begin{remark} \label{koszul_sign}
In connection with the cup product it is common to introduce a Koszul sign convention; see \cite[Theorem III.9.1]{MacLane:1995}. This means, for $d \in \BE^i(Y,X)$ we define morphisms
\begin{equation*}
\begin{aligned}
d_* &\colon \BE^j(W,Y) \to \BE^{i+j}(W,X) \,,&\quad e &\mapsto d \smile e \quad \text{and}\\
d^* &\colon \BE^j(X,Z) \to \BE^{i+j}(Y,Z) \,,&\quad c &\mapsto (-1)^{ij} c \smile d
\end{aligned}
\end{equation*}
that are natural in $W$ and $Z$, respectively. For $i=0$ one has $d_* = \BE^j(W,d)$ and $d^* = \BE^j(d,Z)$. Further, the natural transformations for $i=1$ and $j=0$ extend those in \cite[Definition~3.1]{Nakaoka/Palu:2019}. The Koszul sign convention does not appear in any of the previous works on the cup product of an extriangulated category. 

For a triangulated category $\cat{T}$ the sign seems artificial, but when $\cat{T} = \dcat{\cat{E}}$ is the derived category of an exact category it is standard: An element in $\BE^n(Y,X)$ is considered a degree $n$ morphism $Y \to X$. 
\end{remark}

%%%%%%%%%%%%%%%%%%%%%%%%%%%%%%%%%%%%%%%%%%%%%%%%%%%%%%%%%%%
%%%%%%%%%%%%%%%%%%%%%%%%%%%%%%%%%%%%%%%%%%%%%%%%%%%%%%%%%%%
\section{Biextriangulated functors}
\label{sec-biextria}

We now discuss functors of the form $\cat{A} \times \cat{B} \to \cat{C}$ of extriangulated categories. The notion of a biextriangulated functor should generalise the notions for exact and triangulated categories; see for example \cite[II.7.4]{Weibel:2013} and \cite[Definition~10.3.6]{Kashiwara/Schapira:2006}, respectively. 

%%%%%%%%%%%%%%%%%%%%%%%%%%%%%%%%%%%%%%%%%%%%%%%%%%%%%%%%%%%
\subsection{Extriangulated functors}

An \emph{extriangulated functor} 
\begin{equation*}
\sF = (\sF,\beta) \colon (\cat{A},\BE,\fs) \to (\cat{B},\BF,\ft)
\end{equation*}
of extriangulated categories consists of an additive functor $\sF \colon \cat{A} \to \cat{B}$ and a natural transformation $\beta \colon \BE(-,-) \Longrightarrow \BF(\sF(-),\sF(-))$ such that for any extriangle $X \xrightarrow{f} W \xrightarrow{g} Y \xdashedrightarrow{d}$ in $\cat{A}$ there is an extriangle
\begin{equation*}
\sF(X) \xrightarrow{\sF(f)} \sF(W) \xrightarrow{\sF(g)} \sF(Y) \xdashedrightarrow{\nat{\beta}{Y,X}(d)}
\end{equation*}
in $\cat{B}$; see \cite[Definition~2.32]{BennettTennenhaus/Shah:2021} and also \cite[Definition~2.11]{Nakaoka/Ogawa/Sakai:2022}.

The composition of extriangulated functors $(\sF,\beta) \colon (\cat{A},\BE,\fs) \to (\cat{B},\BF,\ft)$ and  $(\sG,\gamma) \colon (\cat{B},\BF,\ft) \to (\cat{C},\BG,\fu)$ is defined by $(\sG\sF,\alpha)$ where $\nat{\alpha}{Y,X}= \nat{\gamma}{\sF Y,\sF X} \nat{\beta}{YX}$ for all $X,Y\in \cat{A}$; see \cite[Definition~3.18(ii)]{BenTenHaugSandShah}.

Let $(\sF,\beta), (\sG,\gamma) \colon (\cat{A},\BE,\fs) \to (\cat{B},\BF,\ft)$ be extriangulated functors. An \emph{extriangulated natural transformation} $\eta \colon (\sF,\beta) \Longrightarrow (\sG,\gamma)$ is a natural transformation $\eta \colon \sF \Longrightarrow \sG$ of additive functors such that, for all $X,Y \in \cat{A}$, the diagram
\begin{equation*}
\begin{tikzcd}[column sep=large]
\BE(Y,X) \ar[r,"{\nat{\beta}{Y,X}}"] 
\ar[d,"{\nat{\gamma}{Y,X}}" swap] & \BF(\sF(Y),\sF(X)) \ar[d,"{(\nat{\eta}{X})_*}"] \\
\BF(\sG(Y),\sG(X)) \ar[r,"{(\nat{\eta}{Y})^*}" swap] & \BF(\sF(Y),\sG(X))
\end{tikzcd}
\end{equation*}
commutes; see \cite[Definition~2.11]{Nakaoka/Ogawa/Sakai:2022} and also \cite[Definition 4.1]{BenTenHaugSandShah}. 

\begin{remark} \label{nat_trans_cup_product}
The definition of an extriangulated natural transformation can be phrased using the cup product. Explicitly, a natural transformation $\eta \colon (\sF,\beta) \Longrightarrow (\sG,\gamma)$ is extriangulated if and only if
\begin{equation*}
\eta_X \smile \nat{\beta}{Y,X}(d) = \nat{\gamma}{Y,X}(d) \smile \eta_Y
\end{equation*}
for any $d \in \BE(Y,X)$. This means the natural transformation is compatible with the connecting map from $\operatorname{Hom}$ to $\BE$. 
\end{remark}

Specialising to exact or triangulated categories, these notions recover the classical definitions: 

\begin{example} \label{example_exact_functor}
Let $\sF \colon \cat{E} \to \cat{F}$ be an exact functor of exact categories; this means any exact sequence is mapped to an exact sequence. The functor $\sF$ induces a map $\nat{\beta}{Y,X}\colon \Ext[1]{\cat{E}}{Y}{X} \to \Ext[1]{\cat{F}}{\sF(Y)}{\sF(X)}$, which is natural in $X$ and $Y$ since $\sF$ preserves pushout and pullback squares; for the latter see \cite[Proposition~5.2]{Buehler:2010}. If we view $\cat{E}$ and $\cat{F}$ as extriangulated categories as in \cref{example_exact}, then $\sF$, together with this natural transformation $\beta$, is an extriangulated functor.

Conversely, it is clear that any extriangulated functor between exact categories is exact; see \cite[Theorem~2.34]{BennettTennenhaus/Shah:2021}.

Lastly, we note that any natural transformation of exact functors is automatically extriangulated \cite[p.~349]{Nakaoka/Ogawa/Sakai:2022}; see also \cite[Example 5.4]{BenTenHaugSandShah} for details.
\end{example}

\begin{example} \label{example_triangulated_functor}
Let $(\sF,\tau) \colon \cat{S} \to \cat{T}$ be a triangulated functor of triangulated categories; this means $\tau \colon \sF \susp \to\susp \sF$ is a natural transformation and, for any exact triangle $X \xrightarrow{f} X \xrightarrow{g} Y \xrightarrow{h} \susp X$ in $\cat{S}$, there is an exact triangle 
\begin{equation*}
\sF(X) \xrightarrow{\sF(f)} \sF(W) \xrightarrow{\sF(g)} \sF(Y) \xrightarrow{\nat{\tau}{X} \sF(h)} \susp \sF(X)
\end{equation*}
in $\cat{T}$. Then $(\sF,\beta) \colon \cat{S} \to \cat{T}$ is an extriangulated functor of triangulated categories with $\nat{\beta}{Y,X}$ given by the composition
\begin{equation*}
\Hom{\cat{S}}{Y}{\susp X} \xrightarrow{\sF} \Hom{\cat{T}}{\sF(Y)}{\sF(\susp X)} \xrightarrow{\Hom{\cat{T}}{\sF(Y)}{\nat{\tau}{X}}} \Hom{\cat{T}}{\sF(Y)}{\susp \sF(X)}
\end{equation*}
for all $X,Y \in \cat{S}$; see \cite[Theorem~2.33]{BennettTennenhaus/Shah:2021}.

Conversely, let $(\sF,\beta) \colon \cat{S} \to \cat{T}$ be an extriangulated functor between triangulated categories. We set 
\begin{equation*}
\nat{\tau}{X} \coloneqq \nat{\beta}{\susp X,X}(\id_{\susp X}) \colon \sF(\susp X) \to \susp \sF(X)\,.
\end{equation*}
It is straightforward to check that this map is natural in $X$. It is an isomorphism, since it witnesses the mapping of the (ex)triangles
\begin{equation*}
(X \to 0 \to \susp X \xdashedrightarrow{\id_{\susp X}} \susp X) \quad \mapsto \quad (\sF(X) \to 0 \to \sF(\susp X) \xdashedrightarrow{\nat{\tau}{X}} \susp\sF(X))\,.
\end{equation*}
Hence $(\sF,\tau)$ is a triangulated functor. 

Unlike the exact case, a natural transformation of triangulated functors is not necessarily extriangulated. It is extriangulated if and only if it is a morphism of triangulated functors in the sense of \cite[Definition 10.1.9(ii)]{Kashiwara/Schapira:2006}; see \cite[p.~349]{Nakaoka/Ogawa/Sakai:2022} and also \cite[Example 5.3]{BenTenHaugSandShah}.
\end{example}

The next \namecref{lem-prelim-for-extr-functor-higher-ext} is a generalisation of the fact that the higher extensions, defined using syzygies, are functorial from \cite[Proposition~3.4]{Herschend/Liu/Nakaoka:2022}.

\begin{lemma} \label{lem-prelim-for-extr-functor-higher-ext}
Let $(\sF,\beta) \colon (\cat{A},\BE,\fs) \to (\cat{B},\BF,\ft)$ be an extriangulated functor of extriangulated categories. For any morphism $f \colon Y \to Y'$ there exists a morphism $f_\omega \colon \Omega \sF(Y) \to \sF(\Omega Y')$ such that there is a commutative diagram
\begin{equation*}
\begin{tikzcd}
\Omega \sF(Y) \ar[r] \ar[d,"f_\omega"] & Q \ar[r] \ar[d] & \sF(Y)\ar[d,"f"] \ar[r,dashed] & ~ \\
\sF(\Omega Y') \ar[r] & \sF(P) \ar[r] & \sF(Y') \ar[r,dashed] & ~
\end{tikzcd}
\end{equation*}
where each row is an extriangle in $\cat{B}$ and $P$ and $Q$ are $\BE$-projective and $\BF$-projective, respectively. Moreover, the morphism
\begin{equation*}
\BF((\id_Y)_\omega,\id_W) \colon \BF(\sF(\Omega Y),W) \to \BF(\Omega \sF(Y),W)
\end{equation*}
is natural in $W \in \cat{B}$ and $Y \in \cat{A}$. 
\end{lemma}

\begin{proof}
Let $\Omega Y' \to P \to Y' \dashrightarrow$ be an extriangle with $P$ an $\BE$-projective object. Further, let $\Omega \sF(Y) \to Q \to \sF(Y) \dashrightarrow$ be an extriangle with $Q$ a $\BF$-projective object. As $\BF(Q,-) = 0$, applying $\Hom{}{Q}{\sF(-)}$ to the first extriangle yields a short exact sequence by \cite[Corollary~3.12]{Nakaoka/Palu:2019}. That is we obtain a commutative diagram
\begin{equation*}
\begin{tikzcd}
\Omega \sF(Y) \ar[r] & Q \ar[r] \ar[d] & \sF(Y)\ar[d,"f"] \ar[r,dashed] & ~ \\
\sF(\Omega Y') \ar[r] & \sF(P) \ar[r] & \sF(Y') \ar[r,dashed] & ~ \nospacepunct{.}
\end{tikzcd}
\end{equation*}
By the axiom (ET3)$^{\text{op}}$ for an extriangulated category in \cite[Definition~2.12]{Nakaoka/Palu:2019}, there exists a morphism $f_\omega \colon \Omega \sF(Y) \to \sF(\Omega Y)$ such that the diagram commutes. While this morphism need not be unique, the induced map
\begin{equation*}
\BF(f_\omega,\id_W) \colon \BF(\sF(\Omega Y),W) \to \BF(\Omega \sF(Y'),W)
\end{equation*}
is independent of the choice of $f_\omega$. In particular, the two compositions
\begin{equation*}
\Omega \sF(Y) \to \Omega \sF(Y') \to \sF(\Omega Y') \quad \text{and} \quad \Omega \sF(Y) \to \sF(\Omega Y) \to \sF(\Omega Y')
\end{equation*}
need not coincide, however after applying $\BF(-,W)$ they will be the same. That means there is a commutative diagram
\begin{equation*}
\begin{tikzcd}[column sep=huge]
\BF(\Omega \sF(Y),W) \ar[r,"{\BF((\id_Y)_\omega,\id_W)}"] \ar[d,"{\BF(\Omega \sF(f),\id_W)}"] & \BF(\sF(\Omega Y),W) \ar[d,"{\BF(\sF \Omega f,\id_W)}"] \\
\BF(\Omega \sF(Y'),W) \ar[r,"{\BF((\id_{Y'})_\omega,\id_W)}"] & \BF(\sF(\Omega Y'),W)
\end{tikzcd}
\end{equation*}
where the vertical morphisms are well-defined by \cite[Proposition~3.4]{Herschend/Liu/Nakaoka:2022}.
\end{proof}

We record the interplay of extriangulated functors and the higher extensions: 

\begin{proposition} \label{extr_functor_higher_ext}
Let $(\sF,\beta) \colon (\cat{A},\BE,\fs) \to (\cat{B},\BF,\ft)$ be an extriangulated functor of extriangulated categories. Assume higher extensions are well-defined in $(\cat{A},\BE,\fs)$ and $(\cat{B},\BF,\ft)$ through the same construction, i.e. by means of:
\begin{enumerate}
    \item coends following \cref{defn_higher_ext_coend}, assuming $\cat{A}$ and $\cat{B}$ are essentially small; 
    \item syzygies following \cref{higher_ext_syz}, assuming $\cat{A}$ and $\cat{B}$ have enough projectives; or
    \item cosyzygies following \cref{higher_ext_cosyz}, assuming $\cat{A}$ and $\cat{B}$ have enough injectives. 
\end{enumerate}
Then for each $n\geq 0$ there are induced natural transformations
\begin{equation*}
\beta^n \colon \BE^n(-,-) \Longrightarrow \BF^n(\sF(-),\sF(-))
\end{equation*}
such that $\nat{\beta^i}{Y,X}(d) \smile \nat{\beta^j}{Z,Y}(e) = \nat{\beta^{i+j}}{Z,X}(d \smile e)$ for $d \in \BE^i(Y,X)$ and $e \in \BE^j(Z,Y)$. 
\end{proposition}
\begin{proof}
(1) 
%We first assume that the higher extensions in $(\cat{A},\BE,\fs)$ and $(\cat{B},\BF,\ft)$ are defined via coends as described in \cref{higher_ext_coend}. 
Defining higher extensions via coends, we construct the natural transformation $\beta^n$ by induction on $n$. For $n = 0$ it is the map given by  $\sF$ and for $n = 1$ we have $\beta^1 = \beta$. For $n \geq 2$ and any $X,Y,W \in \cat{A}$ we obtain the composition
\begin{equation*}
% \BE^i(W,X) \otimes_R \BE^j(Y,W) \xrightarrow{\phi^i(W,X) \otimes \phi^j(Y,W)} \BF^i(\sF(W),\sF(X)) \otimes_R \BF^j(\sF(Y),\sF(W)) \to \BF^n(\sF(Y),\sF(X))\,.
% \BE^{n-1}(W,X) \otimes_R \BE(Y,W) \to \BF^{n-1}(\sF(W),\sF(X)) \otimes_R \BF(\sF(Y),\sF(W)) \to \BF^n(\sF(Y),\sF(X))\,.
\begin{tikzcd}[column sep=-2em]
\BE^{n-1}(W,X) \otimes_R \BE(Y,W) \ar[dr,"{\nat{\beta^{n-1}}{W,X} \otimes \nat{\beta}{Y,W}}" swap] \ar[rr,dashed] && \BF^n(\sF(Y),\sF(X)) \nospacepunct{.} \\
& \BF^{n-1}(\sF(W),\sF(X)) \otimes_R \BF(\sF(Y),\sF(W)) \ar[ur,"\smile" swap]
\end{tikzcd}
\end{equation*}
The colimit property yields a unique morphism $\nat{\beta^n}{Y,X} \colon \BE^n(Y,X) \to \BF^n(\sF(Y),\sF(X))$, and it is straightforward to check that this map is natural in $X$ and $Y$. By construction the cup product is compatible with these natural transformations. 

(2, 3) 
We define higher extensions via syzygies. The argument for cosyzygies is analogous. Using \cref{lem-prelim-for-extr-functor-higher-ext} $(n-1)$-times we define
\begin{equation*}
\begin{tikzcd}[row sep=small,column sep=huge]
\BE(\Omega^{n-1} Y,X) \ar[r,"{\nat{\beta}{\Omega^{n-1} Y,X}}"] \ar[d,phantom,"\rotatebox{90}{=}"] & \BF(\sF(\Omega^{n-1} Y),\sF(X)) \ar[r] & \BF(\Omega^{n-1} \sF(Y),\sF(X)) \ar[d,phantom,"\rotatebox{90}{=}"] \\
\BE_\Omega^n(Y,X) \ar[rr,dashed,"{\nat{\beta^n}{Y,X}}"] && \BF_\Omega^n(\sF(Y),\sF(X))
\end{tikzcd}
\end{equation*}
for any $n \geq 2$. By \cref{lem-prelim-for-extr-functor-higher-ext} this morphism is natural in $X$ an $Y$.
It remains to check the compatibility of these natural transformations with the cup product. 
There is a canonical and commutative diagram of the form 
\begin{equation*}
\begin{tikzcd}
\Hom{\cat{A}}{\Omega^n Y}{X} \ar[r] \ar[d] & \BE(\Omega^{n-1} Y,X) \ar[d] &[-3em] = &[-3em] \BE_\Omega^n(Y,X) \ar[dd, "{\beta_{Y,X}^{n}}"] \\
\Hom{\cat{B}}{\sF(\Omega^n Y)}{\sF(X)} \ar[r] \ar[d] & \BF(\sF(\Omega^{n-1} Y),\sF(X)) \ar[d] \\
\Hom{\cat{B}}{\Omega^n \sF(Y)}{\sF(X)} \ar[r] & \BF(\Omega^{n-1} \sF(Y),\sF(X)) & = & \BF_\Omega^n(\sF(Y),\sF(X)) \nospacepunct{.}
\end{tikzcd}
\end{equation*}
Hence from here it is straightforward to see that the natural transformation is compatible with the cup product.
\end{proof}

%%%%%%%%%%%%%%%%%%%%%%%%%%%%%%%%%%%%%%%%%%%%%%%%%%%%%%%%%%%
\subsection{Extriangulated adjunction}

Let 
\begin{equation*}
(\sF,\beta) \colon (\cat{A},\BE,\fs) \to (\cat{B},\BF,\ft) \quad \text{and} \quad (\sG,\gamma) \colon (\cat{B},\BF,\ft) \to (\cat{A},\BE,\fs)
\end{equation*}
be extriangulated functors. Following \cite[Definition~4.9(i)]{BenTenHaugSandShah}, we say $(\sG,\gamma)$ is an \emph{extriangulated right adjoint} of $(\sF,\beta)$ and $(\sG,\gamma)$ is an \emph{extriangulated left adjoint}, if $(\sF,\sG)$ is a pair of adjoint functors and the unit and counit witnessing the adjunction are extriangulated natural transformations. In this situation we say $((\sF,\beta),(\sG,\gamma))$ is a \emph{pair of extriangulated adjoint functors}.

Requiring that the unit and counit are extriangulated is superfluous for triangulated category. However, for exact categories it is necessary: For example, if $R$ is a commutative ring and $X$ is a flat non-projective $R$-module, then in the adjunction $-\otimes_{R} X \dashv \Hom{R}{X}{-}$ the functor $-\otimes_{R}$ is exact while $\Hom{R}{X}{-}$ is not. 

A consequence of this definition is that we obtain an adjunction relation on the higher extensions. 

\begin{lemma}
\label{lemma-extriangulated-adjoints}
Let $(\sF,\beta) \colon (\cat{A},\BE,\fs) \to (\cat{B},\BF,\ft)$ be an extriangulated left adjoint of $(\sG,\gamma) \colon (\cat{B},\BF,\ft) \to (\cat{A},\BE,\fs)$ and $\eta$ and $\epsilon$ the associated unit and counit, respectively. Then the maps
\begin{equation*}
\begin{gathered}
\BF(\sF(Y),X) \xrightarrow{\nat{\gamma}{\sF(Y),X}} \BE(\sG(\sF(Y)),\sG(X)) \xrightarrow{(\nat{\eta}{Y})^*} \BE(Y,\sG(X)) \quad \text{and} \\
\BE(Y,\sG(X)) \xrightarrow{\nat{\beta}{Y,\sG(X)}} \BF(\sF(Y),\sF(\sG(X)) \xrightarrow{(\nat{\epsilon}{X})_*} \BF(\sF(Y),X)
\end{gathered}
\end{equation*}
are inverse to each other for any $Y \in \cat{A}$ and $X \in \cat{B}$. In particular, this yields an isomorphism $\BF(\sF(Y),X) \cong \BE(Y,\sG(X))$ for all $Y \in \cat{A}$ and $X \in \cat{B}$ that is compatible with the adjunction of $(\sF,\beta)$ and $(\sG,\gamma)$. 
\end{lemma}
\begin{proof}
We show that the composition $\BF(\sF(Y),X) \to \BE(Y,\sG(X))\to  \BF(\sF(Y),X)$ is the identity map. 
The proof that the composition $\BE(Y,\sG(X))\to \BF(\sF(Y),X) \to\BE(Y,\sG(X)) $ is the identity is similar and omitted. 
Consider the following diagram:
\begin{equation*}
\begin{tikzcd}
\BF(\sF(Y),X) \ar[r,"{\nat{\gamma}{\sF(Y),X}}"] \ar[d,"{(\nat{\epsilon}{\sF(Y)})^*}"] & \BE(\sG(\sF(Y)),\sG(X)) \ar[d,"{\nat{\beta}{\sG(\sF(Y)),\sG(X)}}" swap] \ar[dr,"{(\nat{\eta}{Y})^*}"] \\
\BF(\sF(\sG(\sF(Y))),X) \ar[d,"{\sF(\nat{\eta}{Y})^*}"] & \BF(\sF(\sG(\sF(Y))),\sF(\sG(X))) \ar[l,"{(\nat{\epsilon}{X})_*}" swap] \ar[d,"{\sF(\nat{\eta}{Y})^*}" swap] & \BE(Y,\sG(X)) \ar[dl,"{\nat{\beta}{Y,\sG(X)}}"] \nospacepunct{.} \\
\BF(\sF(Y),X) & \BF(\sF(Y),\sF(\sG(X))) \ar[l,"{(\nat{\epsilon}{X})_*}" swap]
\end{tikzcd}
\end{equation*}
The top left square commutes since $\epsilon$ is an extriangulated natural transformation. 
The bottom left square commutes by virtue of $\BF$ being a bifunctor. 
The triangle on the right commutes as all the maps are natural. It remains to observe that the vertical composition on the left is the identity.
\end{proof}

\begin{corollary} \label{adjunction_compatible_cup}
Let $(\sF,\beta) \colon (\cat{A},\BE,\fs) \to (\cat{B},\BF,\ft)$ be an extriangulated left adjoint of $(\sG,\gamma) \colon (\cat{B},\BF,\ft) \to (\cat{A},\BE,\fs)$. With the isomorphisms  from \cref{lemma-extriangulated-adjoints} the following diagram commutes
\begin{equation*}
\begin{tikzcd}[column sep=large]
\BF(\sF(X),Y') \ar[d,leftrightarrow,"\cong"] & \Hom{\cat{B}}{\sF(X)}{Y} \ar[r,"(\nat{\beta}{X',X}(d))^*"] \ar[l,"e_*" swap] \ar[d,leftrightarrow,"\cong"] & \BF(\sF(X'),Y) \ar[d,leftrightarrow,"\cong"] \\
\BE(X,\sG(Y')) & \Hom{\cat{A}}{X}{\sG(Y')} \ar[r,"d^*"] \ar[l,"(\nat{\gamma}{Y,Y'}(e))_*" swap] & \BE(X',\sG(Y))
\end{tikzcd}
\end{equation*}
for any $d \in \BE(X',X)$ and $e \in \BF(Y,Y')$.
\end{corollary}
\begin{proof}
As the vertical maps are compositions of $\gamma$ and the unit, or $\beta$ and the counit, the claim follows from the definition of an extriangulated functor and \cref{nat_trans_cup_product}.
\end{proof}

From the construction of the higher extensions and \cref{extr_functor_higher_ext} we immediately obtain:

\begin{corollary}
Let $(\sF,\beta) \colon (\cat{A},\BE,\fs) \to (\cat{B},\BF,\ft)$ be an extriangulated left adjoint of $(\sG,\gamma) \colon (\cat{B},\BF,\ft) \to (\cat{A},\BE,\fs)$. We assume that the higher extensions are well-defined in  $(\cat{A},\BE,\fs)$ and $(\cat{B},\BF,\ft)$ through the same construction, in the sense of \Cref{extr_functor_higher_ext}. Then $\BF^n(\sF(Y),X) \cong \BE^n(Y,\sG(X))$ for all $Y \in \cat{A}$ and $X \in \cat{B}$.
Moreover, these isomorphisms are compatible with the cup product similar to \cref{adjunction_compatible_cup}. \qed
\end{corollary}

%%%%%%%%%%%%%%%%%%%%%%%%%%%%%%%%%%%%%%%%%%%%%%%%%%%%%%%%%%%
\subsection{Biextriangulated functors} \label{sec_biextr_functor}

Loosely speaking, a biextriangulated functor is a bifunctor that is extriangulated functor in each component, satisfying an additional compatibility condition between the components.
\Cref{biextri} below only requires an expected compatibility condition; later we discuss a stronger notion.

\begin{definition} \label{biextri}
A \emph{biextriangulated functor} 
\begin{equation*}
(\sF,\phi,\psi) \colon (\cat{A},\BE,\fs) \times (\cat{B},\BF,\ft) \to (\cat{C},\BG,\fu)
\end{equation*}
of extriangulated categories consists of: 
\begin{enumerate}

\item \label{biextri_biadd} a functor $\sF \colon \cat{A} \times \cat{B} \to \cat{C}$; 

\item for each $Y \in \cat{B}$, a natural transformation
\begin{equation*}
\phi^Y \colon \BE(-,-) \to \BG(\sF(-,Y),\sF(-,Y))\,;
\end{equation*}

\item for each $X \in \cat{A}$, a natural transformation
\begin{equation*}
\psi^X \colon \BF(-,-) \to \BG(\sF(X,-),\sF(X,-))\,;
\end{equation*}
\end{enumerate}
such that 
\begin{enumerate}[label=(\alph*), ref=\alph*]

\item \label{biextri_extri_in_2} for each $Y \in \cat{B}$, the functor $(\sF(-,Y), \phi^Y) \colon \mathcal{A} \to \mathcal{C}$ is extriangulated;

\item \label{biextri_extri_in_3} for each $X \in \cat{A}$, the functor $(\sF(X,-), \psi^X) \colon \cat{B} \to \cat{C}$ is extriangulated;

\item \label{biextri_extra_in_B} for $X,X' \in \cat{A}$ and $g \colon Y \to Y'$ in $\cat{B}$, the following diagram commutes
\begin{equation*}
\begin{tikzcd}[column sep=wider]
\BE(X',X) \ar[r,"{\nat{\phi^Y}{X',X}}"] \ar[d,"{\nat{\phi^{Y'}}{X',X}}"] & \BG(\sF(X',Y),\sF(X,Y)) \ar[d,"{\sF(X,g)_*}"] \\
\BG(\sF(X',Y'),\sF(X,Y')) \ar[r,"{\sF(X',g)^*}"] & \BG(\sF(X',Y),\sF(X,Y')) \nospacepunct{;}
\end{tikzcd}
\end{equation*}

\item \label{biextri_extra_in_A} for $f \colon X \to X'$ in $\cat{A}$ and $Y,Y' \in \cat{B}$, the following diagram commutes
\begin{equation*}
\begin{tikzcd}[column sep=wider]
\BF(Y',Y) \ar[r,"{\nat{\psi^X}{Y',Y}}"] \ar[d,"{\nat{\psi^{X'}}{Y',Y}}"] & \BG(\sF(X,Y'),\sF(X,Y)) \ar[d,"{\sF(f,Y)_*}"] \\
\BG(\sF(X',Y'),\sF(X',Y)) \ar[r,"{\sF(f,Y')^*}"] & \BG(\sF(X,Y'),\sF(X',Y)) \nospacepunct{.}
\end{tikzcd}
\end{equation*}
\end{enumerate}
\end{definition}

\begin{remark}
Conditions \cref{biextri_extra_in_B} and \cref{biextri_extra_in_A} mean that $\phi$ and $\psi$ are \emph{extranatural} in $\cat{B}$ and $\cat{A}$, respectively; see \@ \cite{Eilenberg/Kelly:1966}.
\end{remark}

\begin{example} \label{example_exact_bifunctor}
Let $\sF \colon \cat{E} \times \cat{F} \to \cat{G}$ be a \emph{biexact} functor of exact categories; this means that $\sF(X,-)$ and $\sF(-,Y)$ are exact functors for every $X \in \cat{E}$ and $Y \in \cat{F}$; cf.\@ \cite[II.7.4]{Weibel:2013}. We let $\phi^Y$ and $\psi^X$ be the natural transformations obtained by applying $\sF(-,Y)$ and $\sF(X,-)$ to a short exact sequence in $\cat{E}$ and $\cat{F}$, respectively; see \cref{example_exact_functor}. Hence, \cref{biextri_extri_in_2,biextri_extri_in_3} hold. We show that \cref{biextri_extra_in_B} holds without any further assumptions. Let $X \to W \to X'$ be a short exact sequence in $\cat{E}$ and $g \colon Y \to Y'$ a morphism in $\cat{F}$. By \cite[Proposition~3.1]{Buehler:2010}, the morphism of short exact sequences induced by $\sF(-,g)$ factors as follows
\begin{equation*}
\begin{tikzcd}
\sF(X,Y) \ar[r] \ar[d,"{\sF(X,g)}"] & \sF(W,Y) \ar[r] \ar[d] & \sF(X',Y) \ar[d,"="] \\
\sF(X,Y') \ar[r] \ar[d,"="] & Z \ar[r] \ar[d] & \sF(X',Y) \ar[d,"{\sF(X',g)}"] \\
\sF(X,Y') \ar[r] & \sF(W,Y') \ar[r] & \sF(X',Y')
\end{tikzcd}
\end{equation*}
where the top left and bottom right squares are bicartesian. The top row is precisely $\nat{\phi^Y}{X',X}([X \to W \to X'])$ and the bottom row is $\nat{\phi^{Y'}}{X',X}([X \to W \to X'])$. Hence the middle row can be obtained by applying $\Ext[1]{\cat{G}}{\sF(X',Y)}{\sF(X,g)}$ to the top row, or by applying $\Ext[1]{\cat{G}}{\sF(X',g)}{\sF(X,Y')}$ to the bottom row. In particular, \cref{biextri_extra_in_B} holds. An analogous argument shows \cref{biextri_extra_in_A}. So $(\sF,\phi,\psi)$ is a biextriangulated functor. 
 
Conversely, from \cref{example_exact_functor} it is clear that any biextriangulated functor of exact categories is a biexact functor.
\end{example}

\begin{example} \label{example_triangulated_bifunctor}
Let $(\sF,\theta,\zeta) \colon \cat{S} \times \cat{T} \to \cat{U}$ be a \emph{bitriangulated} functor; this means $\theta \colon \sF(\susp -,-) \to \susp \sF(-,-)$ and $\zeta \colon \sF(-,\susp -) \to \susp \sF(-,-)$ are natural transformations such that $(\sF(-,Y),\nat{\theta}{-,Y})$ and $(\sF(X,-),\nat{\zeta}{X,-})$ are triangulated functors for all $X \in \cat{S}$ and $Y \in \cat{T}$, and the following diagram anticommutes. 
\begin{equation} \label{compatibility_triangulated_bifunctor}
\begin{tikzcd}[column sep=large]
\sF(\susp X,\susp Y) \ar[r,"{\nat{\theta}{X,\susp Y}}"] \ar[d,"{\nat{\zeta}{\susp X,Y}}"'] \ar[dr,phantom,"(-1)"] & \susp \sF(X,\susp Y) \ar[d,"{\susp \nat{\zeta}{X,Y}}"] \\
\susp \sF(\susp X,Y) \ar[r,"{\susp \nat{\theta}{X,Y}}"] & \susp^2 \sF(X,Y) \nospacepunct{.}
\end{tikzcd}
\end{equation}
We define the natural transformations $\phi^Y$ and $\psi^X$ using $\nat{\theta}{-,Y}$  and $\nat{\zeta}{X,-}$, respectively, similarly to how $\beta$ was defined in \cref{example_triangulated_functor}. Hence \cref{biextri_extri_in_2,biextri_extri_in_3} hold. In this setting,  condition \cref{biextri_extra_in_B} is  the identification
\begin{equation*}
(\susp \sF(X,g)) \circ \nat{\theta}{X,Y} \circ \sF(f,Y) = \nat{\theta}{X,Y'} \circ \sF(f,Y') \circ \sF(X',g)
\end{equation*}
for all $f \colon X' \to \susp X$ in $\cat{S}$ and $g \colon Y \to Y'$ in $\cat{T}$. This holds as $\theta$, $\sF(f,-)$ and $\sF(-,g)$ are natural transformations. An analogous argument shows \cref{biextri_extra_in_A}. So $(\sF,\phi,\psi)$ is a biextriangulated functor. 
\end{example}

The above argument does not use \cref{compatibility_triangulated_bifunctor}. In fact, a biextriangulated functor of triangulated categories need not induce a bitriangulated functor. We resolve this discrepancy using `strong' biextriangulated functors in the next \namecref{sec:strong-biextriangulated-functors}.

%%%%%%%%%%%%%%%%%%%%%%%%%%%%%%%%%%%%%%%%%%%%%%%%%%%%%%%%%%%
\subsection{Strong biextriangulated functors}
\label{sec:strong-biextriangulated-functors}

We now consider a stronger compatibility condition between the components of a biextriangulated functor. This will resolve the mismatch between biextriangulated and bitriangulated functors of triangulated categories, while not affecting the match between biextriangulated and biexact functors of exact categories. A possible disadvantage of this definition is  that we need to assume that higher extensions exist, however these always exist for triangulated categories. 

\begin{definition} \label{strong_extri_bifunctor}
Let $\sF = (\sF,\phi,\psi) \colon (\cat{A},\BE,\fs) \times (\cat{B},\BF,\ft) \to (\cat{C},\BG,\fu)$ be a biextriangulated functor. We assume that higher extensions are well-defined in $(\cat{C},\BG,\fu)$. We say $\sF$ is \emph{strong}, if for every $d \in \BE(X,X')$ and $e \in \BF(Y,Y')$ one has
\begin{equation} \label{eq:strong_extri_bifunctor}
\nat{\psi^{X'}}{Y,Y'}(e) \smile \nat{\phi^Y}{X,X'}(d) = - \nat{\phi^{Y'}}{X,X'}(d) \smile \nat{\psi^X}{Y,Y'}(e)
\end{equation}
in $\BG^2(\sF(X,Y),\sF(X',Y'))$. 
\end{definition}

\begin{lemma} \label{strong_bifunctor_higher_ext}
Let $\sF = (\sF,\phi,\psi) \colon (\cat{A},\BE,\fs) \times (\cat{B},\BF,\ft) \to (\cat{C},\BG,\fu)$ be a strong biextriangulated functor and assume that higher extensions are well-defined in  each of the categories through the same construction, as in \Cref{extr_functor_higher_ext}. Then
\begin{equation*}
\nat{(\psi^{X'})^j}{Y,Y'}(e) \smile \nat{(\phi^Y)^i}{X,X'}(d) = (-1)^{ij} \nat{(\phi^{Y'})^i}{X,X'}(d) \smile \nat{(\psi^X)^j}{Y,Y'}(e)
\end{equation*}
for every $d \in \BE^i(X,X')$ and $e \in \BF^j(Y,Y')$.
\end{lemma}
\begin{proof}
By the construction in \cref{extr_functor_higher_ext}, the natural transformation $(\phi^Y)^i$ is fully determined by $(\phi^Y)^{\otimes i}$ on
\begin{equation*}
\BE(W_1,X) \otimes_R \BE(W_2,W_1) \otimes_R \cdots \otimes_R \BE(Y,W_{i-1})\,.
\end{equation*}
We obtain the claim by interchanging the $\BE$ and $\BF$ factors of $\BE^i(X,X') \otimes_R \BF^j(Y,Y')$. 
\end{proof}

For exact categories  condition \cref{eq:strong_extri_bifunctor} is superfluous: 

\begin{proposition} \label{lem:biexact-implies-biextriangulated}
Let $\sF \colon \cat{E} \times \cat{F} \to \cat{G}$ be a biexact functor. Then $\sF$ is strong viewed as a biextriangulated functor. In particular, there is a one-to-one correspondence between biexact functors and strong biextriangulated functors of exact categories. 
% Any biexact functor of exact categories $\sF \colon \cat{E} \times \cat{F} \to \cat{G}$ is strong as a biextriangulated functor.
% Therefore, a bifunctor of exact categories is biexact if and only if it is strong biextriangulated.
\end{proposition}
\begin{proof}
Let $d = [X \xrightarrow{f} Y \xrightarrow{g} Z] \in \Ext[1]{\cat{E}}{X}{Z}$ and $e = [X' \xrightarrow{f'} Y' \xrightarrow{g'} Z'] \in \Ext[1]{\cat{F}}{X'}{Z'}$. Then $\sF(d,X') \smile \sF(Z,e)$ is
\begin{equation*}
[\sF(X,X') \xrightarrow{\sF(f,X')} \sF(Y,X') \xrightarrow{\sF(g,f')} \sF(Z,Y') \xrightarrow{\sF(Z,g')} \sF(Z,Z')]
\end{equation*}
and $\sF(X,e) \smile \sF(d,Z')$ is
\begin{equation*}
[\sF(X,X') \xrightarrow{\sF(X,f')} \sF(X,Y') \xrightarrow{\sF(f,g')} \sF(Y,Z') \xrightarrow{\sF(g,Z')} \sF(Z,Z')]\,;
\end{equation*}
see \cref{example_exact_coend}. We take the pushout of the span $\sF(X,Y') \gets \sF(X,X') \to \sF(Y,X')$ and obtain a morphism of exact sequences
\begin{equation*}
\begin{tikzcd}[column sep = 1.2cm]
\sF(X,X') \ar[r,"{\sF(f,X')}"] \ar[d,"{\sF(X,f')}"] & \sF(Y,X') \ar[r,"{\sF(g,X')}"] \ar[d] & \sF(Z,X') \ar[d,"="] \\
\sF(X,Y') \ar[r, "a"] & W \ar[r, "b"] & \sF(Z,X') \nospacepunct{.}
\end{tikzcd}
\end{equation*}
The universal property of the pushout yields the commutative diagram
\begin{equation*}
\begin{tikzcd}[column sep = 1.2cm]
\sF(X,Y') \ar[d] \ar[r, "a"] & W \ar[d, "b"] \ar[r, "c"] & \sF(Y,Y') \ar[d, "{\sF(g,Y')}"] \\
0 \ar[r] & \sF(Z,X') \ar[r, "{\sF(Z,f')}"] \ar[d] & \sF(Z,Y') \ar[d, "{\sF(Z,g')}"] \\
& 0 \ar[r] & \sF(Z,Z')
\end{tikzcd}
\end{equation*}
such that $ca=\sF(f,Y')$, and so the rectangle defined by the top two squares is a pushout square. 
Since $(a,b)$ is a kernel-cokernel pair the top left square is a pushout square. 
By the pasting lemma for pushout squares, every square is a pushout square; see \cite[p.~72,~Exercise~8]{MacLane:1998}. 
So we obtain an extension 
\begin{equation*}
[\sF(X,X') \xrightarrow{\left(\begin{smallmatrix} \sF(X,f') \\ -\sF(f,X') \end{smallmatrix}\right)} \sF(X,Y') \oplus \sF(Y,X') \to \sF(Y,Y') \to \sF(Z,Z')]\,,
\end{equation*}
and a commutative diagram
\begin{equation*}
\begin{tikzcd}[column sep = 1.54cm]
\sF(X,X') \ar[r,"{-\sF(f,X')}"] & \sF(Y,X') \ar[r,"{\sF(g,f')}"] & \sF(Z,Y') \ar[r,"{\sF(Z,g')}"] & \sF(Z,Z') \\
\sF(X,X') \ar[r,"{\left(\begin{smallmatrix} \sF(X,f') \\ -\sF(f,X') \end{smallmatrix}\right)}"] \ar[u,"=" swap] \ar[d,"="] & \sF(X,Y') \oplus \sF(Y,X') \ar[r] \ar[u] \ar[d] & \sF(Y,Y') \ar[r] \ar[u] \ar[d] & \sF(Z,Z') \ar[u,"=" swap] \ar[d,"="] \\
\sF(X,X') \ar[r,"{\sF(X,f')}"] & \sF(X,Y') \ar[r,"{\sF(f,g')}"] & \sF(Y,Z') \ar[r,"{\sF(g,Z')}"] & \sF(Z,Z')
\end{tikzcd}
\end{equation*}
which exhibits the first claim. The second assertion follows from combining this with \cref{example_exact_bifunctor}.
\end{proof}

\begin{proposition} \label{prop:bitriangulated-implies-strong-biextriangulated}
Let $(\sF,\phi,\psi) \colon \cat{S} \times \cat{T} \to \cat{U}$ be a strong biextriangulated functor of triangulated categories. With
\begin{equation*}
\begin{aligned}
\nat{\theta}{X,Y} &\coloneqq \nat{\phi^Y}{\susp X,X}(\id_{\susp X}) \colon \sF(\susp X,Y) \to \susp \sF(X,Y) \quad \text{and} \\
\nat{\zeta}{X,Y} &\coloneqq \nat{\psi^X}{\susp Y,Y}(\id_{\susp Y}) \colon \sF(X,\susp Y) \to \susp \sF(X,Y)
\end{aligned}
\end{equation*}
the functor $(\sF,\theta,\zeta)$ is a bitriangulated functor. 

Moreover, there is a one-to-one correspondence between bitriangulated functors and strong biextriangulated functors on triangulated categories. 
\end{proposition}
\begin{proof}
We assume that $(\sF,\phi,\psi) \colon \cat{S} \times \cat{T} \to \cat{U}$ is a biextriangulated functor of triangulated categories, and that  $\theta$ and $\zeta$ are given as above. From \cref{example_triangulated_functor} we know that $\sF$, together with $\theta$ and $\zeta$, is a triangulated functor in each component. It remains to show \cref{compatibility_triangulated_bifunctor}. By \cref{example_triangulated_coend} this is equivalent to
\begin{equation*}
\nat{\zeta}{X,Y} \smile \nat{\theta}{X,\susp Y} = - \nat{\theta}{X,Y} \smile \nat{\zeta}{\susp X,Y}\,.
\end{equation*}
This is just a specialisation of the definition of a strong biextriangulated functor.

For the second part, it remains to show that the biextriangulated functor that is induced by a bitriangulated functor is strong. This is straightforward to check using \cref{example_triangulated_coend,example_triangulated_functor}.
\end{proof}

We end this section with examples of biextriangulated functors. For derived categories of module categories, the derived tensor product and derived Hom-functor are biextriangulated. However, for module categories with the abelian exact structure the tensor product and the Hom-functor need not be biextriangulated. This can be rectified by restricting to flat modules. 

\begin{example} \label{example-flat-functors}
Let $R$ be a commutative ring. Let $\cat{X}$ be an essentially small $R$-linear category. Let $\lMod{\cat{X}}$ and $\rMod{\cat{X}}$ be the categories of $R$-linear covariant and contravariant functors, respectively, of the form $\cat{X} \to \Mod{R}$. For  $\sF \in \rMod{\cat{X}}$ and $\sG \in \lMod{\cat{X}}$ define an $R$-bilinear functor $\op{\cat{X}}\times\cat{X}\to\Mod{R}$ by $(X,Y)\mapsto \sF(X)\otimes_{R} \sG(Y)$. Taking coends, discussed in \cref{higher_ext_coend}, gives a functor 
\begin{equation*}
\otimes_{\cat{X}}\colon \rMod{\cat{X}}\times \lMod{\cat{X}}\to \Mod{R} \,,\quad (\sF,\sG) \mapsto \int^W \sF(W) \otimes_R \sG(W)\,;
\end{equation*}
see \cite[Section~2]{Fisher:1968}. When $\cat{X}$ is a category with one object, the endomorphism ring is an $R$-algebra $A$ and $\otimes_\cat{X}$ is the classical tensor product of modules over $A$. 

Let $\lFlat{\cat{X}}$ be the full subcategory of $\lMod{\cat{X}}$ consisting of functors $\sF$ such that $\sF \otimes_{\cat{X}} - \colon \rMod{\cat{X}} \to \Mod{R}$ is exact, meaning $\sF$ is \emph{flat} in the sense of \cite[Theorem~3.2(1)]{Oberst-Rohrl:1970}. 
% Similarly, the subcategory $\Flat{\op{\cat{X}}}$ of $\rMod{\cat{X}}$ consisting of functors $\sG$ such that $- \otimes_{\cat{X}} \sG$ is exact. 
The subcategory $\lFlat{\cat{X}}$ (respectively, $\rFlat{\cat{X}}$) is  extension-closed in $\lMod{\cat{X}}$ (respectively,  $\rMod{\cat{X}}$). 
With the abelian exact structure on $\Mod{R}$, one can check that the restriction of $\otimes_{\cat{X}}$ defines a biexact functor 
\begin{equation*}
\rFlat{\cat{X}}\times \lFlat{\cat{X}} \to \Mod{R}\,.
\end{equation*}
Thus this restriction is a strong biextriangulated functor by \cref{example_exact_bifunctor}.  

% A result of Dung and Garc{\'i}a \cite[Theorem 1.1]{Dung-Garcia:2001} says: that $\Flat{\op{\cat{X}}}$ is \emph{locally finitely-presented} as an additive category in the sense of  Crawley-Boevey \cite{Crawley-Boevey:1994}; that every locally finitely-presented additive category arises this way; and that any equivalence of such categories gives rise to a Morita equivalence, in the sense of {\'A}nh and M{\'a}rki \cite{Anh-Marki:1987}, of the corresponding \emph{functor rings}, discussed below. Note that $\Flat{\op{\cat{X}}}$ need not have kernels; even if $\cat{X}$ is locally finitely-presented the existence of kernels was characterised by Crawley-Boevey \cite[p.~1651,~Theorem]{Crawley-Boevey:1994}. 

The functor categories and their subcategories of flat functors are strongly connected to the category of modules over a non-unital $R$-algebra. For a small $R$-linear category $\cat{X}$ the \emph{functor ring} is $A\coloneqq\bigoplus_{X,Y \in \sk(\cat{X})} \Hom{\cat{X}}{X}{Y}$; this ring was  studied in \cite[Chapter~II]{Gabriel:1962}, also see \cite{Garcia--Martinez-Hernandez:1992}. 
This is a possibly non-unital ring with enough idempotents. 
By \cite[II, Proposition~2]{Gabriel:1962}, there is an $R$-linear equivalence between $\lMod{\cat{X}}$ and the category of unital left $A$-modules; one says $M$ is \emph{unital} if $AM=M$. This equivalence restricts to an equivalence between $\lFlat{\cat{X}}$ and the category of flat unital left $A$-modules. Similarly, the category $\rMod{\cat{X}}$ (respectively $\rFlat{\cat{X}}$) is equivalent to the category of (respectively, flat) unital right $A$-modules. Under this equivalence, the functor $\otimes_\cat{X}$ corresponds to the classical tensor product $\otimes_A$; see for example \cite[Lemma~2]{ElKaoutit:2020}. 
\end{example}

%%%%%%%%%%%%%%%%%%%%%%%%%%%%%%%%%%%%%%%%%%%%%%%%%%%%%%%%%%%
\subsection{Stabilisation of extriangulated categories} \label{subsec:stabilisation}

There are a few ways to construct extriangulated categories. One is as a stabilisation of an exact category. We will see that any biextriangulated functor induced on the stabilisation from a biexact functor is strong. 

Let $(\cat{A},\BE,\fs)$ be an extriangulated category. Let $\cat{I}$ be a full subcategory of $\cat{A}$ that is closed under finite direct sums and consists of objects that are $\BE$-projective and $\BE$-injective; for the definition of the latter see \cref{higher_ext_proj_inj}. 
The \emph{stabilisation of $\cat{A}$ along $\cat{I}$}, denoted by $\stab[\cat{I}]{\cat{A}}$, has the same objects as $\cat{A}$ and the morphisms are defined by 
\begin{equation*}
\Hom{\stab[\cat{I}]{\cat{A}}}{X}{Y} \coloneqq \Hom{\cat{A}}{X}{Y}/[\cat{I}](X,Y)\,,
\end{equation*}
where $[\cat{I}](X,Y)$ is the two-sided ideal of morphisms $X \to Y$ in $\cat{A}$ factoring through objects in $\cat{I}$. 

By \cite[Proposition~3.30]{Nakaoka/Palu:2019}, the extriangulated structure on $\cat{A}$ descends to the stabilisation $\stab[\cat{I}]{\cat{A}}$. The extriangulated structure $(\stab{\BE},\stab{\fs})$ on $\stab[\cat{I}]{\cat{A}}$ is as follows: For objects $X,Y \in \stab[\cat{I}]{\cat{A}}$, we have 
\begin{equation*}
\stab{\BE}(Y,X) = \BE(Y,X) \quad \text{and} \quad \stab{\fs}(d) = [X \xrightarrow{[f]} W \xrightarrow{[g]} Y]
\end{equation*}
where $\fs(d) = [X \xrightarrow{f} W \xrightarrow{g} Y]$. 

\begin{lemma} \label{stab_higher_ext}
Let $(\cat{A},\BE,\fs)$ be an extriangulated category and $\cat{I}$ a full subcategory closed under finite direct sums that consists of objects that are $\BE$-projective and $\BE$-injective. If $\cat{A}$ has enough $\BE$-projectives or $\BE$-injectives, then so does $\stab[\cat{I}]{\cat{A}}$. 

Furthermore, whenever higher extensions are well-defined in $(\cat{A},\BE,\fs)$, then they are well-defined in $(\stab[\cat{I}]{\cat{A}},\stab{\BE},\stab{\fs})$ through the same construction, and
\begin{equation*}
\stab{\BE}^n(X,Y) = \BE^n(X,Y)
\end{equation*}
for all $X,Y \in \cat{A}$ and $n \geq 1$. 
\end{lemma}
\begin{proof}
The first claim holds, as under stabilisation $\BE$-projective and $\BE$-injective objects become $\stab{\BE}$-projective and $\stab{\BE}$-injective objects, respectively, and stabilisation preserves extriangles. The second claim is straightforward to check as the construction of the higher extensions commutes with the stabilisation.
\end{proof}

\Cref{stab_functor} below recovers the $n=1 $ case of  \cite[Example~5.9]{BenTenHaugSandShah} where $(\sF,\beta)$ is the identity, $\cat{I}=0$ and $\cat{J}$ consists of all the projective-injectives. 

\begin{lemma} \label{stab_functor}
Let $(\sF,\beta) \colon (\cat{A},\BE,\fs) \to (\cat{B},\BF,\ft)$ be an extriangulated functor. Let $\cat{I}$ and $\cat{J}$ be full subcategories of $\cat{A}$ and $\cat{B}$, respectively, that are closed under finite direct sums and consist of objects that are projective and injective with respect to the respective extriangulated structure. 
If $\sF(\cat{I}) \subseteq \cat{J}$, then $(\sF,\beta)$ induces an extriangulated functor 
\[
(\stab{\sF},\stab{\beta}) \colon (\stab[\cat{I}]{\cat{A}},\stab{\BE},\stab{\fs}) \to (\stab[\cat{J}]{\cat{B}},\stab{\BF},\stab{\ft})\nospacepunct{.}
\]
\end{lemma}
\begin{proof}
By assumption we have $\sF([\cat{I}](X,Y)) \subseteq [\cat{J}](\sF(X),\sF(Y))$. Hence $\sF$ induces an additive functor $\stab{\sF} \colon \stab[\cat{I}]{\cat{A}} \to \stab[\cat{J}]{\cat{B}}$. For the natural transformation $\stab{\beta}$ is well-defined since $\cat{I}$ and $\cat{J}$ consist of projective and injective objects in the respective extriangulated structure. It now follows from the definition of $\stab{\fs}$ and $\stab{\ft}$ that $(\stab{\sF},\stab{\beta})$ is an extriangulated functor.
\end{proof}

\begin{proposition}
    \label{stabilisation_bifunctor}
Let $(\sF,\phi,\psi) \colon (\cat{A},\BE,\fs) \times (\cat{B},\BF,\ft) \to (\cat{C},\BG,\fu)$ be a biextriangulated functor. 
Let $\cat{I}$, $\cat{J}$ and $\cat{K}$ be full subcategories of $\cat{A}$, $\cat{B}$ and $\cat{C}$, respectively, that are closed under finite direct sums and consist of objects that are projective and injective with respect to the respective extriangulated structure. 

If $\sF(\cat{I},\cat{B}) \subseteq \cat{K}$ and $\sF(\cat{A},\cat{J}) \subseteq \cat{K}$, then $(\sF,\phi,\psi) $ induces a biextriangulated functor $(\stab{\sF},\stab{\phi},\stab{\psi}) \colon (\stab[\cat{I}]{\cat{A}},\stab{\BE},\stab{\fs}) \times (\stab[\cat{J}]{\cat{B}},\stab{\BF},\stab{\ft}) \to (\stab[\cat{K}]{\cat{C}},\stab{\BG},\stab{\fu})$. Moreover, if $(\sF,\phi,\psi) $ is strong, then so is the induced functor $(\stab{\sF},\stab{\phi},\stab{\psi})$. 
\end{proposition}
\begin{proof}
The first claim follows from \cref{stab_functor}, and the second claim follows from \cref{stab_higher_ext}.
\end{proof}

Stabilisation of an extriangulated category is motivated by the stabilisation of Frobenius exact categories. An exact category $\cat{E}$ is \emph{Frobenius}, if it has enough projectives and enough injectives and the projectives and injectives coincide. In this situation, the stabilisation $\stab{\cat{E}} = \stab[\cat{I}]{\cat{\cat{E}}}$ with respect to the subcategory of projective-injective objects $\cat{I}$ is triangulated; see \cite[I.2]{Happel:1988}. As every biexact functor is strong when viewed as a biextriangulated functor, such a functor always induces a strong biextriangulated functor on any of its stabilisations by \Cref{stabilisation_bifunctor}. 

\begin{example} \label{example-matrix-factorisation}
Let $S$ be a regular local ring. A \emph{matrix factorisation of $f \in S$} is a pair of maps of free modules $(\Phi \colon F \to G, \Psi \colon G \to F)$ such that
\begin{equation*}
\Phi \Psi = f \id_G \quad \text{and} \quad \Psi \Phi = f \id_F\,.
\end{equation*}
The free modules $F$ and $G$ have the same rank, and when they have rank $n$, we say $(\Phi,\Psi)$ has \emph{size $n$}. Matrix factorisations were introduced in \cite[Section~5]{Eisenbud:1980}; also see \cite[Chapter~7]{Yoshino:1990}. We denote by $\mf[S]{f}$ the category of matrix factorisations of $f$. Then $\mf[S]{f}$ is Frobenius and the only indecomposable projective-injective objects are $(\id_S,f \id_S)$ and $(f \id_S,\id_S)$. The stabilisation $\smf[S]{f}$ is  equivalent as a triangulated category to the stablisation of the category of maximal Cohen--Macaulay modules over $S/(f)$; see \cite[Section~6]{Eisenbud:1980} and \cite[Theorem~7.4]{Yoshino:1990}. 

Suppose $K$ is a field and that $S = K\llbracket x_1,\ldots,x_s \rrbracket$ and $T = K\llbracket y_1, \ldots, y_t \rrbracket$ are power series rings. We fix $f \in S$ and $g \in T$. The tensor product of $(\Phi,\Psi) \in \mf[S]{f}$ and $(\Phi',\Psi')\in \mf[T]{g}$ is defined as
\[
(\Phi,\Psi) \mathrel{\hat{\otimes}} (\Phi',\Psi') \coloneqq
\left( 
\begin{psmallmatrix}
\Phi \otimes \id & \id \otimes \Phi' \\
-\id \otimes \Psi' & \Psi \otimes \id
\end{psmallmatrix}
,
\begin{psmallmatrix}
\Psi \otimes \id & -\id \otimes \Phi' \\
\id \otimes \Psi' & \Phi \otimes \id
\end{psmallmatrix}
\right)\,;
\]
see \cite[Definition~2.1]{Yoshino:1998}. The tensor product is a matrix factorisation of $f+g$ over $S \otimes_K T$, and this yields a functor
\begin{equation*}
\mathrel{\hat{\otimes}} \colon \mf[S]{f} \times \mf[T]{g} \to \mf[S \otimes_K T]{f+g}\,.
\end{equation*}
By \cite[Lemmas~2.2, 2.8]{Yoshino:1998}, this functor is biexact, and, by \cite[Lemma~2.3]{Yoshino:1998}, it induces a bitriangulated functor on the stabilisations. However, as
\begin{equation*}
(\Phi,\Psi) \mathrel{\hat{\otimes}} (\id_{T},g \id_{T}) \cong (\id_{S\otimes_{K}T},(f+g) \id_{S\otimes_{K}T}) \oplus ((f+g) \id_{S\otimes_{K}T},\id_{S\otimes_{K}T})\,,
\end{equation*}
the tensor product does not induce a functor on the intermediate stabilisations. In particular, it does not induce a tensor product on the category $\smf[S]{f}_\cat{I}$ for $\cat{I} = \{(\id,f \id)\}$, which is equivalent to the category of maximal Cohen--Macaulay modules over $S/(f)$; see \cite[Theorem~7.4]{Yoshino:1990}. 

There is another tensor product, called the multiplicative tensor product, defined in \cite[Definition~2.6]{Fomatati:2022}. While it is a biexact functor, it does not respect the projective-injective objects, and hence does not induce a functor on the stabilisation. 
\end{example}

%%%%%%%%%%%%%%%%%%%%%%%%%%%%%%%%%%%%%%%%%%%%%%%%%%%%%%%%%%%
%%%%%%%%%%%%%%%%%%%%%%%%%%%%%%%%%%%%%%%%%%%%%%%%%%%%%%%%%%%
\section{Tensor extriangulated categories}
\label{sec-tensor-extriangulated}

In this section, we introduce tensor extriangulated categories, generalising the notion of a tensor triangulated category; cf.\@ \cite[Appendix~A.2]{Hovey/Palmieri/Strickland:1997}.

%%%%%%%%%%%%%%%%%%%%%%%%%%%%%%%%%%%%%%%%%%%%%%%%%%%%%%%%%%%
\subsection{Monoidal categories}

Recall that a \emph{monoidal category} $(\cat{A},\otimes,\unit)$ is a category $\cat{A}$ together with a functor
\begin{equation*}
\otimes \colon \cat{A} \times \cat{A} \to \cat{A},
\end{equation*}
called the \emph{tensor product}, and a \emph{unit} object $\unit \in \cat{A}$ equipped with isomorphisms
\begin{gather*}
\nat{\alpha}{X,Y,Z} \colon X \otimes (Y \otimes Z) \overset{\cong}{\longrightarrow} (X \otimes Y) \otimes Z \,, \\
\nat{\lambda}{X} \colon \unit \otimes X \overset{\cong}{\longrightarrow} X \quad\text{and}\quad \nat{\rho}{X} \colon X \otimes \unit \overset{\cong}{\longrightarrow} X
\end{gather*}
that are natural in each argument, satisfying some coherence axioms. A monoidal category is \emph{symmetric} if, in addition, there exists an isomorphism
\begin{equation*}
\nat{\sigma}{X,Y} \colon X \otimes Y \overset{\cong}{\longrightarrow} Y \otimes X\,,
\end{equation*}
also natural in each argument, satisfying some coherence axioms. For a detailed definition see \cite[XI.1]{MacLane:1998}, and \cite{Kelly:1964} for minimally sufficient coherent axioms. 

%%%%%%%%%%%%%%%%%%%%%%%%%%%%%%%%%%%%%%%%%%%%%%%%%%%%%%%%%%%
\subsection{Tensor extriangulated categories}

We now have all the ingredients to define a monoidal extriangulated category. Roughly speaking, it is an extriangulated category with a monoidal structure that respects the extriangulated structure. Explicitly, this means:

\begin{definition}
\label{definition-tensor-extriangulated-category}
We call a tuple 
$(\cat{A},\BE,\fs, \otimes, \unit)$
a \emph{monoidal extriangulated category} if:
\begin{enumerate}
\item \label{item:TE-category-monoidal} $(\cat{A},\otimes, \unit)$ is a monoidal category;
\item \label{item:TE-category-extriangulated} $(\cat{A},\BE,\fs)$ is an extriangulated category; 
\item \label{item:TE-category-tensor} $\otimes = (-\otimes-,\phi,\psi)\colon (\cat{A},\BE,\fs)\times (\cat{A},\BE,\fs)\to (\cat{A},\BE,\fs)$ is a biextriangulated functor; and 
\item \label{item:TE-category-alpha-lambda-rho-extriangulated} the natural isomorphisms $\alpha,\lambda, \rho$ from the monoidal structure are extriangulated natural transformations in each variable. 
\end{enumerate}
We say a monoidal extriangulated category $(\cat{A},\BE,\fs,\otimes,\unit)$ is \emph{strong}, if the biextriangulated functor $(- \otimes -,\phi,\psi)$ is strong (see \Cref{strong_extri_bifunctor}). 

A \emph{tensor extriangulated category} is a monoidal extriangulated category where the monoidal structure is symmetric and the corresponding natural transformation $\sigma$ is extriangulated in each variable. 
\end{definition}

Condition~\cref{item:TE-category-alpha-lambda-rho-extriangulated} has various consequences. For example, it yields
\begin{equation} \label{tec_phi_unit}
\nat{\phi^\unit}{X,Y} = (\nat{\rho^{-1}}{Y})_* \circ (\nat{\rho}{X})^* \quad \text{and} \quad \nat{\psi^\unit}{X,Y} = (\nat{\lambda^{-1}}{Y})_* \circ (\nat{\lambda}{X})^*\,.
\end{equation}
In particular, $\phi^\unit$ and $\psi^\unit$ are natural isomorphisms. Further, in a tensor extriangulated category, the natural transformations $\phi$ and $\psi$ from the tensor product determine each other; explicitly
\begin{equation*}
(\nat{\sigma}{X',Y})_* \circ \nat{\phi^Y}{X,X'} = (\nat{\sigma}{X,Y})^* \circ \nat{\psi^Y}{X,X'}
\end{equation*}
as $\sigma$ is an extriangulated natural isomorphism in each variable.

%%%%%%%%%%%%%%%%%%%%%%%%%%%%%%%%%%%%%%%%%%%%%%%%%%%%%%%%%%%%%
\subsection{Ring of higher extensions}

Recall from \cref{higher_ext_ring_bimod} that $\BE^*(X,X)$ with the cup product is a graded ring. Similarly as for triangulated categories, a tensor structure makes the graded endomorphism ring graded-commutative.

\begin{lemma} \label{extension_ring_commutative}
Let $(\cat{A},\BE,\fs, \otimes, \unit)$ be a strong monoidal extriangulated category. We assume that higher extensions are well-defined in $(\cat{A},\BE,\fs)$. Then $\BE^*(\unit,\unit)$ with the cup product is a graded-commutative graded ring. 
\end{lemma}
\begin{proof}
Let $d \in \BE^i(\unit,\unit)$ and $e \in \BE^j(\unit,\unit)$. Then
\begin{equation*}
\begin{aligned}
d \smile e &= \nat{(\phi^\unit)^i}{\unit,\unit}(d) \smile \nat{(\psi^\unit)^j}{\unit,\unit}(e) \\
&= (-1)^{ij} \nat{(\psi^\unit)^j}{\unit,\unit}(e) \smile \nat{(\phi^\unit)^i}{\unit,\unit}(d) \\
&= (-1)^{ij} e \smile d
\nospacepunct{.}
\end{aligned}
\end{equation*}
For the first and third equalities we implicitly use a higher version of \cref{tec_phi_unit}. 
The second equality follows from \cref{strong_bifunctor_higher_ext}. 
\end{proof}

\begin{example} \label{example:sanity-check}
In \cref{example_triangulated} we explained that any triangulated category is naturally an extriangulated category. By \cref{lem:biexact-implies-biextriangulated} a strong biextriangulated functor on a triangulated category is a bitriangulated functor. Hence the definition of a strong tensor extriangulated category recovers the notion of a tensor triangulated category in the sense of \cite[A.2]{Hovey/Palmieri/Strickland:1997}. Note, that other sources use a weaker definition of a tensor triangulated category, they might not assume that the monoidal structure is strong or the natural transformations $\alpha$, $\rho$, $\lambda$ and $\sigma$ are extriangulated.

The extension ring $\BE^*(\unit,\unit)$ of a tensor triangulated category is precisely the non-negatively graded part of the endomorphism ring $\End[*]{\cat{T}}{\unit}$. 
\end{example}

\begin{remark}
In an essentially small extriangulated category negative extensions were defined in \cite[Definition~5.1]{Gorsky/Nakaoka/Palu:2021}. For a triangulated category these yield the negative part of the endomorphism ring $\End[*]{\cat{T}}{\unit}$; see \cite[Proposition~5.4]{Gorsky/Nakaoka/Palu:2021}. We expect that \cref{extension_ring_commutative} extends to $\BE^*(X,X) = \{\BE^n(X,X)\}_{n \in \BZ}$.
\end{remark}

We will see more examples of (strong) tensor extriangulated categories in \cref{sec:example_tec}. 

%%%%%%%%%%%%%%%%%%%%%%%%%%%%%%%%%%%%%%%%%%%%%%%%%%%%%%%%%%%
\subsection{Closed tensor extriangulated categories}

A symmetric monoidal category $(\cat{A},\otimes,\unit)$ is \emph{closed} in the sense of \cite[p.~119]{Lewis/May/Steinberger:1986} if there is an \emph{internal hom}
\begin{equation*}
\hom{-}{-} \colon \op{\cat{A}} \times \cat{A} \to \cat{A}\,,
\end{equation*}
meaning a bifunctor such that, for every $X \in \cat{A}$, the functor $\hom{X}{-}$ is right adjoint to $- \otimes X$.

\begin{definition}
\label{definition:closed}
We say a tensor extriangulated category $(\cat{A},\BE,\fs,\otimes,\unit)$ is \emph{closed}, if for every $X \in \cat{A}$ the functor $- \otimes X$ has an extriangulated right adjoint $\hom{X}{-}$, and $\hom{-}{-}$ is a biextriangulated functor. 
\end{definition}

We do not know whether the internal hom functor of a closed and strong tensor extriangulated category is strong. 

%%%%%%%%%%%%%%%%%%%%%%%%%%%%%%%%%%%%%%%%%%%%%%%%%%%%%%%%%%%
\subsection{Tensor extriangulated functors}

The following definition  generalises the notion of a  tensor triangulated functor  from \cite[Definition~3]{Balmer:2010}. 

\begin{definition}
A \emph{tensor extriangulated functor} 
\begin{equation*}
(\sF,\beta) \colon (\cat{A},\otimes_\cat{A},\unit_\cat{A}) \to (\cat{B},\otimes_\cat{B},\unit_\cat{B})
\end{equation*}
of tensor extriangulated categories consists of: 
\begin{enumerate}
\item an extriangulated functor $(\sF,\beta) \colon (\cat{A},\BE,\fs) \to (\cat{B},\BF,\ft)$; 
\item an isomorphism $u \colon \sF(\unit_\cat{A}) \to \unit_\cat{B}$; and
\item a natural isomorphism $\nat{\mu}{X,Y} \colon \sF(X \otimes_\cat{A} Y) \to \sF(X) \otimes_\cat{B} \sF(Y)$
\end{enumerate}
such that the following diagrams commute:
\begin{equation*}
\begin{gathered}
\begin{tikzcd}
\sF(\unit_\cat{A} \otimes_{\cat{A}} X) \ar[r,"{\nat{\mu}{\unit_\cat{A},X}}"] \ar[d,"{\sF(\nat{\lambda}{X})}"] & \sF(\unit_\cat{A}) \otimes_\cat{B} \sF(X) \ar[d,"{u \otimes_\cat{B} \sF(X)}"] \\
\sF(X) & \unit_\cat{B} \otimes_\cat{B} \sF(X) \ar[l,"\nat{\lambda}{\sF(X)}"]
\end{tikzcd}
\quad \text{and} \\
\begin{tikzcd}
\BE(X,X') \ar[r,"{\nat{\beta}{X,X'}}"] \ar[d,"{\nat{\phi^Y}{X,X'}}"] & \BF(\sF(X),\sF(X')) \ar[d,"{\nat{\phi^{\sF(Y)}}{\sF(X),\sF(X')}}"] \\
\BE(X \otimes_\cat{A} Y,X' \otimes_\cat{A} Y) \ar[d,"{\nat{\beta}{X \otimes_\cat{A} Y,X' \otimes_\cat{A} Y}}"] & \BF(\sF(X) \otimes_\cat{B} \sF(Y), \sF(X') \otimes_\cat{B} \sF(Y)) \ar[d,"(\nat{\mu}{X,Y})^*"] \\
\BF(\sF(X \otimes_\cat{A} Y),\sF(X' \otimes_\cat{A} Y)) \ar[r,"{(\nat{\mu}{X',Y})_{*}}"] & \BF(\sF(X \otimes_\cat{A} Y),\sF(X') \otimes_\cat{B} \sF(Y)) \nospacepunct{.}
\end{tikzcd}
\end{gathered}
\end{equation*}
\end{definition}

%%%%%%%%%%%%%%%%%%%%%%%%%%%%%%%%%%%%%%%%%%%%%%%%%%%%%%%%%%%
%%%%%%%%%%%%%%%%%%%%%%%%%%%%%%%%%%%%%%%%%%%%%%%%%%%%%%%%%%%
\section{Examples of tensor extriangulated categories} \label{sec:example_tec}

By construction the notion of a tensor extriangulated category generalises that of a tensor triangulated category. In comparison, there seem to be multiple different notions  of a tensor exact category. Often the monoidal structure on an exact or abelian category is left exact, but not necessarily exact. In \cref{example-flat-functors} we defined a biexact functor by restricting the tensor product to flat functors. 
In \Cref{example_flat_bifunctors} we obtain a monoidal structure, which is not symmetric, by extending the tensor product to bimodules. 

\begin{example} \label{example_flat_bifunctors}
Let $R$ be a commutative ring and $\cat{X}$ a small $R$-linear category. Let $\biMod{\cat{X}}{\cat{X}}$ be the category of $R$-bilinear functors $\op{\cat{X}}\times \cat{X}\to \Mod{R}$. We extend the functor $\otimes_\cat{X}$ of \cref{example-flat-functors} to a functor
\begin{equation*}
\begin{gathered}
\otimes_\cat{X} \colon \biMod{\cat{X}}{\cat{X}}\times \biMod{\cat{X}}{\cat{X}}\to \biMod{\cat{X}}{\cat{X}} \,, \\
(\sF,\sG) \mapsto \left((X,Y) \mapsto \int^W \sF(W,Y) \otimes_R \sG(X,W)\right)\,.
\end{gathered}
\end{equation*}
It is straightforward to check that $\biMod{\cat{X}}{\cat{X}}$ equipped with $\otimes_{\cat{X}}$ is a monoidal category and $\Hom{\cat{X}}{-}{-}$ serves as the monoidal unit; see \cite[Proposition~1]{ElKaoutit:2020}. Now consider the subcategory $\Flat{\cat{X},\cat{X}}$ of $\biMod{\cat{X}}{\cat{X}}$ consisting of $R$-bilinear functors $\sF$ such that $- \otimes_{\cat{X}} \sF$ and $\sF \otimes_\cat{X} -$ are exact. Then $\Flat{\cat{X},\cat{X}}$ equipped with $\otimes_\cat{X}$ and $\Hom{\cat{X}}{-}{-}$ is a monoidal extriangulated category. 

When $A$ is the functor ring of $\cat{X}$ (see \cref{example-flat-functors}), then there is an equivalence of monoidal categories  $\biMod{\cat{X}}{\cat{X}}\to \Mod{\op{A} \otimes_R A}$ under which $\otimes_{\cat{X}}$ corresponds to $\otimes_A$; see \cite[Proposition~2]{ElKaoutit:2020}. 
% We have not defined monoidal extriangulated functors. Maybe we should, but at the moment we cannot say there is an equivalence of monoidal extriangulated categories.
\end{example}

The monoidal structure in the previous \namecref{example_flat_bifunctors} is generally not symmetric. One obtains a symmetric monoidal structure for commutative rings:

\begin{example}
\label{example-local-units-flat}
Let $R$ be a commutative ring \emph{with local units}; this means that for any $r_{1},\dots,r_{n}\in R$, there exists an idempotent $e\in R$ such that $er_{i}=r_{i}=r_{i}e$ for each $i$. The standard tensor product $\otimes_R$ is a symmetric monoidal structure on the category of unital $R$-modules $\Mod{R}$; an $R$-module $M$ is \emph{unital} provided $M=RM$. However, the tensor product is generally only left exact. An $R$-module $F$ is \emph{flat}, if $F \otimes_R -$ is exact. The subcategory of flat $R$-modules, denoted $\Flat{R}$, is an extension-closed subcategory of $\Mod{R}$, and hence inherits an exact structure. Hence $(\Flat{R},\otimes_R,R,\operatorname{Ext}_R^1)$ is a tensor extriangulated category. The ring of extensions of the unit $R$ is $\BE^*(R,R) = \Ext[*]{R}{R}{R} = R$. 
\end{example}

A big class of extriangulated categories, that are neither triangulated or exact, are the extension-closed subcategories of triangulated categories; see \cite[Remark~2.18]{Nakaoka/Palu:2019}. In general, extension-closed subcategories of a tensor triangulated category need not inherit the monoidal structure. However, in some situations we obtain a tensor extriangulated category. 

\begin{example}
\label{example-derived-cat-t-structures}
Let $R$ be a commutative ring. The derived category $\dcat{\Mod{R}}$ of $R$-modules is a tensor triangulated category with tensor product $\lotimes_R$ and unit $R$. The standard t-structure $(\cat{U},\cat{V})$ on $\dcat{\Mod{R}}$ is given by 
\begin{equation*}
\begin{gathered}
\cat{U} \coloneqq \set{X \in \dcat{\Mod{R}} | \hh[<0]{X} = 0} \quad \text{and} \\
\cat{V} \coloneqq \set{X \in \dcat{\Mod{R}} | \hh[>0]{X} = 0}\,,
\end{gathered}
\end{equation*}
the subcategories of complexes with homology concentrated in non-negative and non-positive  degrees, respectively. As $\cat{U}$ and $\cat{V}$ are extension-closed subcategories of $\dcat{\Mod{R}}$, they inherit an extriangulated structure. Further, the aisle $\cat{U}$ is closed under the tensor product; that is $\cat{U} \lotimes_R \cat{U} \subseteq \cat{U}$ and $R \in \cat{U}$. Hence $(\cat{U},\lotimes_R,R)$ is a tensor extriangulated category. 

One gets further examples by taking the intersection of $\cat{U}$ with the full subcategory of right-bounded  complexes. When $R$ is noetherian, then the intersection of $\cat{U}$ with the right-bounded derived category of finitely generated modules yields a tensor extriangulated category as well. 
\end{example}

%%%%%%%%%%%%%%%%%%%%%%%%%%%%%%%%%%%%%%%%%%%%%%%%%%%%%%%%%%%
\subsection{Stabilisation of extriangulated categories} \label{subsec:stabilisation_ttc}

In \cref{subsec:stabilisation} we gave conditions when a (strong) biextriangulated bifunctor induces a (strong) biextriangulated bifunctor on the stabilisation with respect to a subcategory of projective-injective objects. We now focus on monoidal extriangulated categories. We say an extension-closed subcategory $\cat{I}$ of a monoidal extriangulated category $(\cat{A},\otimes,\unit)$ is a \emph{two-sided tensor ideal}, if $\cat{A} \otimes \cat{I} \subseteq \cat{I}$ and $\cat{I} \otimes \cat{A} \subseteq \cat{I}$. Then the following result is a direct consequence of \cref{stabilisation_bifunctor}:

\begin{lemma}
\label{lemma:stabilisation_extriangulated}
Let $(\cat{A},\otimes,\unit)$ be a monoidal extriangulated category. If $\cat{I}$ is a two-sided tensor ideal consisting of objects that are $\BE$-projective and $\BE$-injective, then $\underline{\cat{A}}_\cat{I}$ inherits the monoidal structure form $\cat{A}$. In particular, $(\underline{\cat{A}}_\cat{I},\underline{\otimes},\unit)$ is a monoidal extriangulated category. When $(\cat{A},\otimes,\unit)$ is tensor extriangulated then so is $(\underline{\cat{A}}_\cat{I},\underline{\otimes},\unit)$. When $(\cat{A},\otimes,\unit)$ is strong, then so is $(\underline{\cat{A}}_\cat{I},\underline{\otimes},\unit)$. \qed
\end{lemma}

\begin{lemma}
\label{lemma-closed-tensor-implies-projs-and-injs-ideal}
    Let $(\cat{A},\BE,\fs,\otimes,\unit)$ be a closed tensor extriangulated category. 
    %(as in \cref{definition:closed}). 
    %
    Then the $\BE$-projectives form a tensor ideal. 
\end{lemma}

\begin{proof}
    Let $X,Y\in \cat{A}$. 
    We are assuming $-\otimes X$ has an extriangulated right adjoint $\hom{X}{-}$. 
    By \cref{lemma-extriangulated-adjoints} this means $\BE(Y \otimes X,-)\cong \BE(Y,\hom{X}{-})$. 
    To say that an object $Z$ is $\BE$-projective is equivalent to saying that  $\BE(Z,-)=0$. 
    Altogether we have shown that if $Y$ is $\BE$-projective then so is  $Y \otimes X$. 
    Since we are working in a tensor extriangulated category, the tensor product is symmetric, and so we have shown that the $\BE$-projectives form a tensor ideal. 
\end{proof}

\begin{example} \label{example_Hopf}
Let $H$ be a cocommutative Hopf algebra over a field $k$. For $H$-modules $M$ and $N$ the comultiplication induces an $H$-module structure on $M \otimes_k N$. By \cite[Chapter~12, Proposition~3]{Margolis:1983}, the tensor product is an exact, symmetric monoidal structure on $\Mod{H}$. Hence $(\Mod{H},\otimes_k,k,\operatorname{Ext}_H^1)$ is a tensor extriangulated category. 
% Alternate citation for monoidal structure: \cite[Proposition~III.5.1]{Kassel:1995}

Further, a cocommutative Hopf algebra $H$ over a field $k$ is a Frobenius algebra by \cite{Larson/Sweedler:1969}. Hence $\Mod{H}$, as well as the full subcategory of finite-dimensional $H$-modules $\mod{H}$, is Frobenius exact. It is well-known that $H \otimes_k M$ is free for every $H$-module $M$; see for example \cite[Chapter~12, Proposition~4]{Margolis:1983}. Hence, every non-trivial tensor ideal of $\Mod{H}$ contains all projective objects. In particular, there are exactly two tensor ideals of $\Mod{H}$ consisting of projective-injective objects, namely the zero subcategory and the full subcategory of projective-injective objects. Hence, the strong tensor extriangulated category $\Mod{H}$ induces one other strong tensor extriangulated category via stabilisation, namely the tensor triangulated category $\sMod{H}$. 

A special case of a finite-dimensional cocommutative Hopf algebra is a group algebra $kG$ for a finite group $G$. 
\end{example}

%%%%%%%%%%%%%%%%%%%%%%%%%%%%%%%%%%%%%%%%%%%%%%%%%%%%%%%%%%%
\subsection{Dualisable objects}

\cref{example_Hopf} is an instance of an abelian \emph{rigid} monoidal category, that is an abelian monoidal category in which all objects are dualisable; see \cite[p.~254,~Examples]{Ulbrich:1990}. We show that the existence of dualisable objects makes the category Frobenius. 

Let $(\cat{A},\otimes,\unit)$ be a closed monoidal category. The \emph{dualisation functor} is given by
\begin{equation*}
(-)^\vee \coloneqq \hom{-}{\unit} \colon \op{\cat{A}} \to \cat{A}\,.
\end{equation*}
An object $X$ in $\cat{A}$ is called \emph{strongly dualisable}, if the natural transformation
\begin{equation*}
X^\vee \otimes Y \to \hom{X}{Y}
\end{equation*}
is an isomorphism; see \cite[Chapter~III, \S1]{Lewis/May/Steinberger:1986}. When the monoidal product is symmetric, then there exists a natural transformation $X \to (X^\vee)^\vee$. This is an isomorphism if $X$ is strongly dualisable by \cite[Proposition~1.3]{Lewis/May/Steinberger:1986}. 
Note that if $(\cat{A},\BE,\fs,\otimes,\unit)$ is closed tensor extriangulated and if $X$ is strongly dualisable then the natural isomorphisms $X^\vee \otimes - \to \hom{X}{-}$ and $X\otimes -\to (X^{\vee})^{\vee}\otimes -$ are compositions of extriangulated natural transformations, and hence extriangulated. 

% Recall that if  $V$ is strongly dualisable then so is  $V^{\vee}$ and the adjunction gives isomorphisms $\Hom{\cat{A}}{U\otimes V}{W}\to \Hom{\cat{A}}{U}{W\otimes V^{\vee}}$ sending $g\colon U\otimes V\to W$ to the map $U\to W\otimes V^{\vee}$ defined by composition of the natural isomorphism $U\cong  U\otimes \unit$, the image $U\otimes \unit \to U\otimes V\otimes V^{\vee}$ of the \emph{coevaluation map} $\unit \to V\otimes V^{\vee}$ under $U\otimes -$ from \cite[Definition~1.1]{Lewis/May/Steinberger:1986}, and $g\otimes V^{\vee}$; see \cite[Proposition~1.2]{Lewis/May/Steinberger:1986}.

\begin{proposition}
    \label{prop:closed-with-str-dual-objs-make-Frob}
    Let $(\cat{A},\BE,\fs,\otimes,\unit)$ be a closed tensor extriangulated category.
    If each object is strongly dualisable, then $\BE$-projectives are the same as $\BE$-injectives.  
\end{proposition}
\begin{proof}
Let $X,Y\in \cat{A}$.
Since $X$ is strongly dualisable and $\cat{A}$ is symmetric we have 
%$X\cong (X^{\vee})^{\vee}$ which means 
\[
\BE(Y,X)\cong 
\BE(Y,\hom{X^{\vee}}{\unit})\cong 
\BE (Y\otimes X^{\vee},\unit)\cong 
\BE (X^{\vee}\otimes Y,\unit)
\cong
\BE (X^{\vee},Y^{\vee}) \,;
\]
the second and fourth  isomorphisms hold by \cref{lemma-extriangulated-adjoints}. 
Hence $X$  is $\BE$-injective if and only if $X^{\vee}$ is $\BE$-projective, and  $Y$ is $\BE$-projective if and only if $Y^{\vee}$ is $\BE$-injective. 

Let $U \to V \to W \dashrightarrow$ be an extriangle in $\cat{A}$. Since $X$ is strongly dualisable, the functor $\hom{X^{\vee}}{-} \cong (X^\vee)^\vee \otimes - \cong X \otimes -$ is an extriangulated right adjoint to $-\otimes X^{\vee}$. Combining the adjunction with the unitor we obtain a diagram
\[
\begin{tikzcd}[column sep=small]
    \Hom{\cat{A}}{X^{\vee}}{U}
    \arrow[r]\arrow[d]
    &
    \Hom{\cat{A}}{X^{\vee}}{V}
    \arrow[r]\arrow[d]
    &
    \Hom{\cat{A}}{X^{\vee}}{W}
    \arrow[r]\arrow[d]
    &
    \BE(X^{\vee},U)
    \arrow[d]
    \\
     \Hom{\cat{A}}{\unit}{X\otimes U}
     \arrow[r]
    &
    \Hom{\cat{A}}{\unit}{X\otimes V}
    \arrow[r]
    &
    \Hom{\cat{A}}{\unit}{X\otimes W}
    \arrow[r]
    &
    \BE(\unit,X\otimes U)
\end{tikzcd}    
\]
which commutes and in which every row is exact and every vertical map is an isomorphism; see \cref{adjunction_compatible_cup} and \cite[Proposition~(2-ii)]{Nakaoka/Palu:2019}. 

Suppose $X$ is $\BE$-projective. 
By \cref{lemma-closed-tensor-implies-projs-and-injs-ideal} we have that $X\otimes W$ is $\BE$-projective, and so the morphism $\Hom{\cat{A}}{\unit}{X\otimes W}
        \to 
        \BE(\unit,X\otimes U)$ in the diagram above is zero, meaning that the map $\Hom{\cat{A}}{X^{\vee}}{V}
        \to 
        \Hom{\cat{A}}{X^{\vee}}{W}$ is surjective. 
        Altogether, for any extriangle $U \to V \to W \dashrightarrow$ every morphism $X^{\vee}\to W$ factors through the deflation $V \to W$. 
        Thus  $X^{\vee}$ is $\BE$-projective, and hence $X$ is $\BE$-injective. 

        If instead $X$ is $\BE$-injective, then $X^{\vee}$ is $\BE$-projective by the first argument above, and so $X^{\vee}$ is $\BE$-injective by the second, and so $(X^{\vee})^{\vee}$ is $\BE$-projective by the first, and $X\cong (X^{\vee})^{\vee}$ since $\otimes$ is symmetric and $X$ is strongly dualisable. 
\end{proof}

\begin{corollary}
Let $(\cat{A},\BE,\fs,\otimes,\unit)$ be a closed tensor extriangulated category in which each object is strongly dualisable. If $\unit$ has a $\BE$-projective cover or a $\BE$-injective envelope, then $\cat{A}$ is Frobenius.
\end{corollary}
\begin{proof}
Thanks to \cref{prop:closed-with-str-dual-objs-make-Frob} it is enough to check that $\cat{A}$ has enough $\BE$-projectives and $\BE$-injectives. If $P \to \unit$ is a $\BE$-projective cover of $\unit$ then $P \otimes X \to \unit \otimes X \cong X$ is a $\BE$-projective cover of $X$. Hence $\cat{A}$ has enough $\BE$-projectives. Further, the dual $\unit^\vee \cong \unit \to P^\vee$ of the $\BE$-projective cover $P \to \unit$ is a $\BE$-injective cover. Hence by a similar argument as above $\cat{A}$ has enough $\BE$-injectives.
\end{proof}

\begin{remark}
In a closed symmetric monoidal category the dual $X^\vee$ of a strongly dualisable object $X$ is a right and left dual of $X$. In this setting, assuming that all objects are strongly dualisable is equivalent to assuming the category is rigid. However, for an extriangulated tensor category that is rigid, it is not clear whether the functor given by the left or right duals is an extriangulated functor. In the case of abelian categories this can be resolved using the universal properties of kernels and cokernels; see \cite[Proposition~4.2.9]{Etingof/Gelaki/Nikshych/Ostrik:2015}. 
\end{remark}

%%%%%%%%%%%%%%%%%%%%%%%%%%%%%%%%%%%%%%%%%%%%%%%%%%%%%%%%%%%
\subsection{Extriangulated substructure} 
\label{sec:relative-structure-to-generate-new-examples-from-old}

Let $(\cat{A},\BE,\fs)$ be an extriangulated category and $\cat{I}$ a full subcategory of $\cat{A}$. Then there is a closed subfunctor of $\BE$ defined by
\[
\BE_{\cat{I}}(Y,X) \coloneqq \{d \in\BE(Y,X)\mid f^*(d)=0\text{ for all } W \in\cat{I}\text{ and }f\in\Hom{\cat{A}}{W}{Y}\}\nospacepunct{;} 
\]
see \cite[Proposition~3.16]{Herschend/Liu/Nakaoka:2021}. So, $(\cat{A},\BE_\cat{I},\fs_\cat{I})$ is an extriangulated category by \cite[Proposition~3.19]{Herschend/Liu/Nakaoka:2021} where $\fs_\cat{I}$ is the restriction of $\fs$. 

\begin{lemma}
\label{lemma:generating-new-examples-from-old}
Let $(\cat{A},\BE,\fs, \otimes, \unit)$ be a monoidal extriangulated category and let $\cat{I}$ be a full subcategory of $\cat{A}$. Suppose that, for the biextriangulated functor $(-\otimes-,\phi,\psi)$, 
\begin{equation*}
\nat{\phi^{Y}}{X',X}(d) \in \BE_{\cat{I}}(X'\otimes Y,X \otimes Y)
\quad \text{and} \quad
\nat{\psi^{X}}{Y',Y}(e)\in \BE_{\cat{I}}(X\otimes Y',X\otimes Y)
\end{equation*}
for each $X,X',Y,Y' \in\cat{A}$, and each $d \in \BE_{\cat{I}}(X',X)$ and $e \in \BE_{\cat{I}}(Y',Y)$. Then $(\cat{A},\BE_{\cat{I}},\fs_{\cat{I}}, \otimes, \unit)$ is a monoidal extriangulated category. 

When $(\cat{A},\BE,\fs, \otimes, \unit)$ is symmetric or strong, then so is $(\cat{A},\BE_{\cat{I}},\fs_{\cat{I}}, \otimes, \unit)$.
\end{lemma}

\begin{proof}
By \cite[Proposition~3.16]{Herschend/Liu/Nakaoka:2021} we have that $\BE_{\cat{I}}$ is a subfunctor of $\BE$. By our assumption  the restriction of  $\varphi^{Y}$ to $\BE_{\cat{I}}(-,-)$ defines a subtransformation in the sense of \cite[p.~261]{Eilenberg-Maclane:1945}. We call this restriction $\varphi^{Y}_\cat{I}$. We also obtain a restriction $\psi^X_\cat{I}$ of $\psi^X$. We need to check conditions \cref{item:TE-category-monoidal,item:TE-category-extriangulated,item:TE-category-tensor,item:TE-category-alpha-lambda-rho-extriangulated} from \cref{definition-tensor-extriangulated-category} hold. Condition \cref{item:TE-category-monoidal} holds by assumption. Condition \cref{item:TE-category-extriangulated} holds by the discussion before the statement of the lemma. In what remains we check conditions \cref{item:TE-category-tensor,item:TE-category-alpha-lambda-rho-extriangulated}. 

We start with \cref{item:TE-category-tensor} and we check \cref{biextri_extri_in_2} from \cref{biextri} holds; that is that the functor $(- \otimes Y,\phi^Y_\cat{I})$ is extriangulated. Suppose $X \to X'' \to X' \xdashedrightarrow{d}$ is an extriangle in $(\cat{A},\BE_\cat{I},\fs_\cat{I})$ with $d \in \BE(X',X)$. Then it is an extriangle in $(\cat{A},\BE,\fs)$ and, as $(- \otimes Y,\phi^Y)$ is an extriangulated functor, 
\begin{equation*}
X\otimes Y \to X''\otimes Y \to X'\otimes Y \xdashedrightarrow{\phi^Y_\cat{I}(d)}
\end{equation*}
is an extriangle. As $\fs_\cat{I}$ is the restriction of $\fs$ to $\BE_\cat{I}$ the functor $(- \otimes Y,\phi^Y_\cat{I})$ is extriangulated. Using a similar argument shows that \cref{biextri_extri_in_3} from \cref{biextri} holds; that is $(X \otimes -,\psi^X_\cat{I})$ is extriangulated. 

To finish the proof we need to show that certain diagrams in $\Mod{R}$  commute; namely the compatibility diagrams of $\phi_\cat{I}$ and $\psi_\cat{I}$ from \cref{biextri_extra_in_B} and \cref{biextri_extra_in_A} of \cref{biextri}, and their compatibility with $\alpha, \rho, \lambda$ of the monoidal structure. Each diagram is the restriction of the corresponding diagram for $(\cat{A},\otimes,\unit)$. Hence they commute.

When $(\cat{A},\otimes,\unit)$ is symmetric, then so is $(\cat{A},\otimes,\unit)$ by the same argument as above. For strongness, note that the cup product is compatible with restriction. 
\end{proof}

%%%%%%%%%%%%%%%%%%%%%%%%%
\subsection{Pure-exact extriangulated structure}
\label{subsec-pure-exact}

In the next example we discuss a special case of extriangulated subfunctors for compactly generated triangulated categories. The extriangles arising are precisely the pure-exact triangles.

We recall some basic definitions concerning triangulated categories. Let $\cat{T}$ be a triangulated category with suspension $\susp$. An object $X$ is \emph{compact} if the functor $\Hom{\cat{T}}{X}{-}\colon \cat{T}\to\abcat$ preserves coproducts, and we write $\compact{\cat{T}}$ for the full subcategory of compact objects. The triangulated category $\cat{T}$ is \emph{compactly generated}, if it has small coproducts, $\compact{\cat{T}}$ is essentially small and for any non-zero object $Y$ in $\cat{T}$ there exists an object $X\in \compact{\cat{T}}$ such that $\Hom{\cat{T}}{X}{Y}\neq 0$. 

% Recall  \cref{example_triangulated} and \cref{example:sanity-check}. 
\begin{proposition}
\label{closed-compact-gen-tri-makes-pure-exact-relative-extriangulated}
Let $(\cat{T},\otimes,\unit)$ be a compactly generated closed tensor triangulated category, considered as a tensor extriangulated category $(\mathcal{T},\BE,\fs, \otimes, \unit)$.  In the notation of \cref{sec:relative-structure-to-generate-new-examples-from-old}, we have that $(\mathcal{T},\BE_{\compact{\cat{T}}},\fs_{\compact{\cat{T}}}, \otimes, \unit)$ is a tensor extriangulated category. 
\end{proposition}

We first recall the definition and some basic properties of phantom maps from \cite{Krause:2000} as they are closely connected to the extriangulated substructure $\BE_{\compact{\cat{T}}}$. 

Let $\rMod{(\compact{\cat{T}})}$ be the category of additive functors $\op{(\compact{\cat{T}})}\to \abcat$. This is an abelian AB5 category. Let $\mathbf{Y}\colon \cat{T} \to \rMod{(\compact{\cat{T}})}$ be the \emph{restricted Yoneda functor}, defined on an object $X$ by the restriction 
of $\Hom{\cat{T}}{-}{X}$ to $\compact{\cat{T}}$ and likewise on morphisms. An additive functor $\sF \colon \cat{T}\to \rMod{(\compact{\cat{T}})}$ is \emph{cohomological} if any exact triangle 
\begin{equation}
\label{eq-triangle}
\begin{tikzcd}
X  \ar[r,"f"]  & Y \ar[r,"g"] & Z \ar[r,"h"] & \susp X
\end{tikzcd}
\end{equation}
induces an exact sequence $\sF(X) \to \sF(Y)\to \sF(Z)\to \sF(\susp X)$ in $\abcat$. In particular, the restricted Yoneda functor $\mathbf{Y}$ is cohomological. 

A morphism $f$ in $\cat{T}$ is called \emph{phantom} if $\mathbf{Y}(f)=0$. The composition of a phantom morphism with any morphism is phantom. A triangle \eqref{eq-triangle} is said to be \emph{pure-exact} if $0\to \mathbf{Y}(X)\to \mathbf{Y}(Y)\to \mathbf{Y}(Z)\to 0$ is exact. So \eqref{eq-triangle} is pure-exact if and only if $h$ is phantom; see \cite[Lemma~1.3]{Krause:2000}. 

\begin{lemma}
\label{lemma-on-phantom}
Let $(\cat{T},\otimes,\unit)$ be a closed tensor triangulated category. Let $f$ and $g$ be morphisms in $\cat{T}$. If $f$ or $g$ is phantom, then so is $f \otimes g$. 
\end{lemma}

\begin{proof}
As phantom morphisms are closed under composition with any other morphism, it is enough to show $f \otimes \id$ is phantom when $f$ is phantom and $\id \otimes g$ is phantom when $g$ is phantom. We give a proof for the former, the latter holds by symmetry.

The restricted Yoneda functor $\mathbf{Y}$ preserves coproducts. Since $- \otimes X$ is a left adjoint for any object $X \in \cat{T}$, it also preserves coproducts. It follows that the composition $\sF \coloneqq \mathbf{Y}\circ (- \otimes X)$ is a cohomological functor $\cat{T}\to \rMod{(\compact{\cat{T}})}$ that preserves coproducts. 

Let $f$ be a phantom morphism. By \cite[Corollary 2.5]{Krause:2000} we have that $\mathbf{Y}(f \otimes \id_X) = \sF(f)=0$ as $\sF$ is cohomological and preserves coproducts. 
\end{proof}

We can now give an explicit description of the extriangulated substructure $\BE_{\compact{\cat{T}}}$ from \cref{closed-compact-gen-tri-makes-pure-exact-relative-extriangulated}. If $d \in \BE(Y,X) = \Hom{\cat{T}}{Y}{\susp X}$, then $f^*(d) = d \circ f$ for every $f \colon W \to Y$ with in $\cat{T}$. Hence $d \in \BE_{\compact{\cat{T}}}(Y,X)$ if and only $d \circ f = 0$ for every map $f \colon C \to Y$ with $C \in \compact{\cat{T}}$. That means $\BE_{\compact{\cat{T}}}(Y,X)$ is the subset of phantom maps in $\Hom{\cat{T}}{Y}{\susp X}$, and $X \to Y \to Z \dashrightarrow$ is an extriangle if and only if $X \to Y \to Z \to \susp X$ is a pure-exact triangle.

\begin{example}
    Let $R$ be a non-artinian ring.  
    The derived category  $\dcat{\Mod{R}}$ is compactly generated triangulated where compact objects are perfect complexes. 
    We claim that the extriangulated structure on $\dcat{\Mod{R}}$, given by \cite[Proposition~3.19]{Herschend/Liu/Nakaoka:2021} taking $\cat{I}$ to be the compact objects, is neither  triangulated nor exact.  

Using that $R$ is not artinian, it cannot be pure-semisimple by the implication $(i)\Rightarrow(a)$ in \cite[Theorem~2.1]{Prest:1984}. 
By \cite[Corollary~7.2]{Garkusha/Prest:2005}, it follows that $\dcat{\Mod{R}}$ is not pure-semisimple, and so there exists an object $X$ that is not pure-injective.  
By \cite[Definition~1.1]{Krause:2000}, there exists a non-split pure-exact triangle $X\to Y\to Z\to \susp X$. 
Since categorical monics split  in triangulated categories, $X\to Y$ cannot be monic.  
If the extriangulated structure on $\dcat{\Mod{R}}$ were exact, the inflation  $X\to Y$ would be a kernel, which is false. 

It suffices to find a contradiction assuming the extriangulated structure is triangulated. 
Choose a morphism $X\to Y$ in $\dcat{\Mod{R}}$ by concentrating an arbitrary homomorphism $f$ in $\Mod{R}$ in degree $0$. 
By assumption we can complete $X\to Y$ to a pure-exact triangle $X\to Y\to Z\to \susp X$. 
By \cite[Lemma~2.4]{Garkusha/Prest:2005} we have that $0\to \hh[n]{X}\to \hh[n]{Y}\to \hh[n]{Z}\to 0$ is pure-exact in $\Mod{R}$ for each $n\in\BZ$. 
Taking $n=0$ implies that $f$ is a pure-monomorphism. 
Since pure-monomorphisms  in  $\Mod{R}$ are injective, a contradiction is found by choosing $f$ not injective.  
\end{example}

Finally we prove of \cref{closed-compact-gen-tri-makes-pure-exact-relative-extriangulated} using the observations in \cref{sec:relative-structure-to-generate-new-examples-from-old}.

\begin{proof}[Proof of \cref{closed-compact-gen-tri-makes-pure-exact-relative-extriangulated}.]
We check that the assumptions of \cref{lemma:generating-new-examples-from-old} are satisfied. By \cref{example_triangulated_functor,example_triangulated_bifunctor} we have
\begin{equation*}
\nat{\phi^Y}{X',X}(d) = (X' \otimes Y \xrightarrow{d \otimes \id_Y} (\susp X) \otimes Y \xrightarrow{\cong} \susp (X \otimes Y))
\end{equation*}
for $d \in \BE(X',X)$. If $d$ is phantom, then so is $\nat{\phi^Y}{X',X}(d)$ by \cref{lemma-on-phantom}. The analogous condition for $\psi$ can be shown similarly. Hence $(\cat{T},\BE_{\compact{\cat{T}}},\fs_{\compact{\cat{T}}},\otimes,\unit)$ is an extriangulated tensor category by \cref{lemma:generating-new-examples-from-old}.
\end{proof}

% To give an example of this, let $\Lambda$ be a derived discrete algebra. Then $\cat{T}=\cat{K}(\Proj{\Lambda})$, the unbounded homotopy category of complexes of projective modules satisfies the desired conditions; see \cite{Arnesen/Laking/Pauksztello/Prest:2017}. More generally one can take $\Lambda$ to be any gentle algebra; see \cite{Bennett-Tennenhaus:2019}. 

%%%%%%%%%%%%%%%%%%%%%%%%%%%%%%%%%%%%%%%%%%%%%%%%%%%%%%%%%%%
%%%%%%%%%%%%%%%%%%%%%%%%%%%%%%%%%%%%%%%%%%%%%%%%%%%%%%%%%%%
\section{Classification of tensor ideals} \label{sec:tensor_ideals}

The goal of this \namecref{sec:tensor_ideals} is the classification of the radical thick tensor ideals in an essentially small tensor extriangulated category. 
In  \cite[Sections~7~and~8]{Buan/Krause/Solberg:2005} lattices of thick tensor ideals were considered  for both abelian and  triangulated categories simultaneously. 
The arguments for the classification of thick tensor ideals in a triangulated category of \cite{Balmer:2005} and \cite{Kock/Pitsch:2017} generalise straight away to the extriangulated case. For convenience we give the details of some important parts. 

%%%%%%%%%%%%%%%%%%%%%%%%%%%%%%%%%%%%%%%%%%%%%%%%%%%%%%%%%%%
\subsection{Tensor ideals}

Let $(\cat{A},\otimes,\unit)$ be a symmetric monoidal category. A non-empty strictly full subcategory $\cat{I}$ of $\cat{A}$ is a \emph{tensor ideal}, if for each $X \in \cat{A}$ and $Y \in \cat{I}$, we have $X \otimes Y \in \cat{I}$. 

For an object $X \in \cat{A}$ and a positive integer $n$, we set $X^{\otimes n} \coloneqq X \otimes \dots \otimes X$, the $n$-fold tensor product of $X$. For convenience we also set $X^{\otimes 0} \coloneqq \unit$. 

We say a tensor ideal $\cat{I}$ is:
\begin{enumerate}
 \item \emph{radical} if, 
 for all $X\in \cat{A}$, we have that $X^{\otimes n} \in \cat{I}$ for some $n \geq 1$ implies $X\in \cat{I}$; and 
 \item \emph{prime} if, for all $X,Y\in \cat{A}$, we have that $X\otimes Y \in\cat{I}$ implies $X\in\cat{I}$ or $Y\in\cat{I}$.
\end{enumerate}
Note that any prime tensor ideal is radical. 

Let $(\cat{A},\BE,\fs)$ be an extriangulated category. We say a non-empty full subcategory $\cat{I}$ is \emph{thick}, if it is closed under direct summands and has the \emph{2-out-of-3 property}; that is, for any extriangle  $X \to Y \to Z \dashrightarrow$,  if any two of $X$, $Y$, or $Z$ lie in $\cat{I}$, then so does the third. 

Let $\cat{C}$ be a collection of objects of $\cat{A}$. We denote by $\thick^\otimes(\cat{C})$ the smallest thick tensor ideal containing $\cat{C}$, and by $\rad(\cat{C})$ the smallest radical thick tensor ideal containing $\cat{C}$. These categories exist as the intersection of (radical) thick tensor ideals is again a (radical) thick tensor ideal. As taking the thick tensor closure is a finite process, the radical closure can be described as
\begin{equation} \label{radical_closure}
\rad(\cat{C}) = \set{X \in \cat{A} | X^{\otimes n} \in \thick^\otimes(\cat{C}) \text{ for some } n \geq 1}\,.
\end{equation}
For a tensor extriangulated category $(\mathcal{A},\otimes, \unit)$ we denote by $\Rad(\cat{A})$ the class of all radical thick tensor ideals of $\cat{A}$. When $\cat{A}$ is essentially small, then $\Rad(\cat{A})$ is a set. In \Cref{Rad_coherent} we show that, if $\Rad(\cat{A})$ is a set, it has a rich poset structure. 

%%%%%%%%%%%%%%%%%%%%%%%%%%%%%%%%%%%%%%%%%%%%%%%%%%%%%%%%%%%
\subsection{Lattices and frames}

For completeness we recall the terminology and fix notation; we loosely follow \cite{Johnstone:1982}. Recall, in a poset $L$ the \emph{join} (respectively, \emph{meet}) of a subset $S \subseteq L$ is the smallest (respectively, largest) element that contains (respectively, is contained in) every element of $S$. We denote the join of $S$ by $\bigvee S$ and the meet of $S$ by $\bigwedge S$. For elements $a,b \in L$ we write $a \vee b = \bigvee \{a,b\}$ and $a \wedge b = \bigwedge \{a,b\}$. 

A \emph{lattice} is a poset $L$ in which every finite, non-empty subset has a join and a meet. A lattice $L$ is \emph{distributive} if $L$ has a \emph{greatest element} $1 = \bigwedge \emptyset$, a \emph{least element} $0 = \bigvee \emptyset$ and it satisfies the \emph{distributive law}; that is
\begin{equation*}
a \wedge (b \vee c) = (a \wedge b) \vee (a \wedge c) \,, \quad \text{or equivalently} \quad a \vee (b \wedge c) = (a \vee b) \wedge (a \vee c)
\end{equation*}
for any $a,b,c \in L$. A \emph{morphism of distributive lattices} is a map that respects the order, finite meets, finite joins, and the greatest and least element.

A lattice $F$ is a \emph{frame} if the join of any subset exist and it satisfies the \emph{infinite join-distributivity law}; that is
\begin{equation*}
a \wedge \bigvee S = \bigvee \set{a \wedge s | s \in S}
\end{equation*}
for any $a \in F$ and any subset $S \subseteq F$. A \emph{morphism of frames} is a morphism of distributive lattices that preserves arbitrary joins. 

%%%%%%%%%%%%%%%%%%%%%%%%%%%%%%%%%%%%%%%%%%%%%%%%%%%%%%%%%%%
\subsection{Stone and Hochster duality}

The prototypical example of a frame is the set of all open subsets of a topological space. In fact, there is a functor
\begin{equation*}
\Omega \colon \Topcat \to \op{\Frmcat}
\end{equation*}
from the category of topological spaces with continuous maps to the category of frames. This functor has a right adjoint, given by the prime spectrum of a frame $(F,\leq)$: An element $a \in F$ is \emph{(meet-)prime}, if $b \wedge c\leq a $ implies that either $a \leq b$ or $a \leq c$. The \emph{prime spectrum of $F$} is the set of all prime elements in $F$ endowed with a topology given by the closed sets $V(a) = \set{p | a \leq p}$. We denote the right adjoint by $\pt$, due to the fact that prime elements in a frame correspond to points; for more details see \cite[II.1.3]{Johnstone:1982}. 

The adjoint functors $(\Omega,\pt)$ restrict to an equivalence between the category of sober spaces and the category of spatial frames. This is called \emph{Stone duality}; for more details see \cite[II.1.7]{Johnstone:1982}. This equivalence restricts further to an equivalence of the category of spectral spaces and the category of coherent frames. We recall the latter two notions.

An element $a$ in a frame $F$ is \emph{compact} if for each subset $S$ of $F$ with $a \leq \bigvee S$ there is a finite subset $T \subseteq S$ such that $a \leq \bigvee T$. A frame $F$ is \emph{coherent} if the greatest element $1$ is compact, $a \wedge b$ is compact for any compact elements $a,b \in F$, and for each $a \in F$ there exists a subset $S$ of compact elements satisfying $a = \bigvee S$. A \emph{morphism of coherent frames} is a morphism of frames that preserves compact elements. 

In a coherent frame $F$ the set of compact elements is a distributive sublattice of $F$, and the frame $F$ can be recovered from the distributive lattice of compact elements. In fact, there is an equivalence $\Cohcat \to \DLatcat$ from the category of coherent frames to the category of distributive lattices; see \cite[II.3.2]{Johnstone:1982}. 

On the other hand, a topological space $X$ is \emph{spectral}, if 
\begin{enumerate}
\item $X$ is quasi-compact;
\item the quasi-compact open subsets of $X$ are closed under finite intersections and form a basis of open sets of the topology on $X$; and
\item $X$ is \emph{sober}; that is every non-empty irreducible closed subset has a generic point; 
\end{enumerate}
see \cite{Hochster:1969}. A \emph{morphism of spectral spaces} is a continuous map such that the preimage of every quasi-compact open set is again quasi-compact. We denote the category of spectral spaces by $\Spcat$. The spectral spaces are precisely those topological spaces that occur as the Zariski spectrum of a commutative ring; see \cite{Hochster:1969}.

Putting everything together, we obtain a sequence of equivalences:
\begin{equation*}
\begin{tikzcd}
\Spcat \ar[r,shift left,"\Omega"] & \op{\Cohcat} \ar[l,shift left,"\pt"] \ar[r,"\compact{(-)}"] & \op{\DLatcat} \ar[r,leftrightarrow,"\op{(-)}"] & \op{\DLatcat} & \op{\Cohcat} \ar[l,"\compact{(-)}" swap] \ar[r,shift right,"\pt" swap] & \Spcat \ar[l,shift right,"\Omega" swap] \nospacepunct{.}
\end{tikzcd}
\end{equation*}
Here, the functor $\op{(-)}$ sends a lattice to its opposite lattice, which has the same underlying set and the reversed order. Note, that this is not a contravariant functor. The composition of these equivalences yields a functor $\Spcat \to \Spcat$, called \emph{Hochster duality}. This equivalence is usually denoted by $(-)^\vee$, and $(X^\vee)^\vee = X$ for a spectral space $X$.

%%%%%%%%%%%%%%%%%%%%%%%%%%%%%%%%%%%%%%%%%%%%%%%%%%%%%%%%%%%
\subsection{Lattice of tensor ideals}

Let $(\cat{A},\otimes,\unit)$ be a tensor extriangulated category. The meet of two radical thick tensor ideals is given as their intersection. The join of a set of radical thick tensor ideal $S$ is
\begin{equation*}
\bigvee S = \rad(\bigcup_{\cat{I} \in S} \cat{I})\,.
\end{equation*}
The least element is the nilradical ideal $\rad(\{0\})$ and the greatest element is the whole category $\cat{A}$. Using the same arguments as for the triangulated case in \cite[Theorem~3.1.9]{Kock/Pitsch:2017}, we show the following:

\begin{proposition} \label{Rad_coherent}
Suppose $(\mathcal{A},\otimes, \unit)$ is a tensor extriangulated category. If $\Rad(\cat{A})$ is a set, then it is a coherent frame.
\end{proposition}
\begin{proof}
We first show the infinite join-distributive law. Let $\cat{I} \in \Rad(\cat{A})$ and $S \subseteq \Rad(\cat{A})$. We show $\cat{I} \wedge \bigvee S \leq \bigvee_{\cat{J} \in S} (\cat{I} \wedge \cat{J})$, as the opposite inclusion always hold. For $X \in \cat{I} \wedge \bigvee S$ we define
\begin{equation*}
\cat{J}_X \coloneqq \set{Y \in \bigvee S | X \otimes Y \in \bigvee_{\cat{J} \in S} (\cat{I} \wedge \cat{J})}\,.
\end{equation*}
It is straightforward to check that $\cat{J}_X$ is a radical thick tensor ideal. As $\cat{J}_X \geq \cat{J}$ for all $\cat{J} \in S$ we get
\begin{equation*}
\cat{J}_X \geq \bigvee S \ni X\,.
\end{equation*}
Hence $X \otimes X \in \bigvee_{\cat{J} \in S} (\cat{I} \wedge \cat{J})$, and as $\in \bigvee_{\cat{J} \in S} (\cat{I} \wedge \cat{J})$ is radical, this means $X\in \bigvee_{\cat{J} \in S} (\cat{I} \wedge \cat{J})$. This shows $\Rad(\cat{A})$ is a frame. 
Note that the compact objects in $\Rad(\cat{A})$ are all of the form $\rad(X)$ for an object $X \in \cat{A}$. 

Let $X \in \cat{A}$ and assume $\rad(X) \subseteq \bigvee S$ for some $S \subseteq \Rad(\cat{A})$. Then $X^{\otimes n} \in \thick^\otimes(\bigcup_{\cat{J} \in S} \cat{J})$. Taking the thick tensor closure is a finite process, hence there exists a finite subset $T \subseteq S$ such that $X^{\otimes n} \in \thick^\otimes(\bigcup_{\cat{J} \in T} \cat{J})$. Hence $X \in \bigvee T$. 

For coherence it remains to show that the meet of $\rad(X)$ and $\rad(Y)$ for any $X,Y \in \cat{A}$ is $\rad(X \otimes Y)$. Since $\rad(X)$ is a radical thick tensor ideal containing $X$ it also contains $X\otimes Y$, and so $\rad(X) \geq \rad(X\otimes Y)$. Similarly $\rad(Y) \geq \rad(X\otimes Y)$. Thus $\rad(X) \wedge \rad(Y) \geq \rad(X\otimes Y)$, and we now require the reverse inclusion. Let $U \in \rad(X) \wedge \rad(Y)$. Then $U^{\otimes 2} \in \rad(X \otimes Y)$ and $U \in \rad(X \otimes Y)$ by \cref{radical_closure}. This finishes the proof.
\end{proof}

%%%%%%%%%%%%%%%%%%%%%%%%%%%%%%%%%%%%%%%%%%%%%%%%%%%%%%%%%%%
\subsection{Classification}

For the classification it remains to apply Stone and Hochster duality to \cref{Rad_coherent}. For tensor triangulated categories this result is due to Balmer \cite{Balmer:2005}. The lattice-theoretic proof relies on work by \cite{Kock/Pitsch:2017} and \cite{Buan/Krause/Solberg:2005}. However, first we describe the prime elements of $\Rad(\cat{A})$. 

\begin{lemma} 
Let $(\mathcal{A},\otimes, \unit)$ be a tensor extriangulated category such that $\Rad(\cat{A})$ is a set. A radical thick tensor ideal $\cat{P}$ is prime in $\Rad(\cat{A})$ if and only if $\cat{P}$ is a prime thick tensor ideal. 
\end{lemma}
\begin{proof}
We assume $\cat{P}$ is prime as an element in the lattice $\Rad(\cat{A})$. Suppose $X\otimes Y \in \cat{P}$ for some objects $X,Y\in\cat{A}$. Then 
\begin{equation*}
\cat{P} \geq \rad(X \otimes Y) = \rad(X) \wedge \rad(Y)
\end{equation*}
using an intermediate claim proven in \cref{Rad_coherent}. By assumption, we have that $\rad(X) \leq \cat{P}$ or $\rad(Y) \leq \cat{P}$. Hence $X \in \cat{P}$ or $Y \in \cat{P}$. 

Conversely, we assume that $\cat{P}$ is a prime thick tensor ideal. Let $\cat{I},\cat{J}\in\Rad(\cat{A})$ such that $\cat{I}\cap \cat{J} \leq \cat{P}$ and $\cat{I} \not\leq \cat{P}$. We need to show $\cat{J} \leq \cat{P}$. Let $X \in \cat{I}$ with $X \notin \cat{P}$. Then $X \otimes Y \in \cat{P}$ for every $Y \in \cat{J}$. As $\cat{P}$ is a prime thick tensor ideal, this yields $Y \in \cat{P}$ for every $Y \in \cat{J}$. Hence $\cat{J} \leq \cat{P}$. 
\end{proof}

\begin{theorem}
\label{theorem-full-radical-bijection}
Let $(\mathcal{A},\otimes, \unit)$ be a tensor extriangulated category such that $\Rad(\cat{A})$ is a set. We let $\balmer{\cat{A}}$ be the set of all prime thick tensor ideals in $\cat{A}$ and endow it with a topology given by the closed sets 
\begin{equation*}
\supp(X) \coloneqq \set{\cat{P} \in \balmer{\cat{A}} | X \notin \cat{P}}
\end{equation*}
for any $X \in \cat{A}$. Then there is a bijection
\begin{gather*}
\Rad(\cat{A}) \leftrightarrow \set{\text{Thomason closed subsets of } \balmer{\cat{A}}}, \\
\cat{I} \mapsto \supp(\cat{I}) = \bigcup_{X \in \cat{I}} \supp(X) \quad \text{and} \quad \set{X \in \cat{A} | \supp(X) \subseteq V} \mapsfrom V\,.
\end{gather*}
\end{theorem}
\begin{proof}
As the prime thick tensor ideals are precisely the prime elements of $\Rad(\cat{A})$, we have
\begin{equation*}
\balmer{\cat{A}} = (\pt(\Rad(\cat{A})))^\vee
\end{equation*}
as topological spaces. The classification follows from the description of the functors $\pt$ and $(-)^\vee$.
\end{proof}

The topological space $\balmer{\cat{A}}$ is called the \emph{Balmer spectrum} of $\cat{A}$.

%%%%%%%%%%%%%%%%%%%%%%%%%%%%%%%%%%%%%%%%%%%%%%%%%%%%%%%%%%%%
\subsection{Functorality}

Let $(\sF,\beta) \colon (\cat{A},\otimes_\cat{A},\unit_\cat{A}) \to (\cat{B},\otimes_\cat{B},\unit_\cat{B})$ be a tensor extriangulated functor. Then
\begin{gather*}
\rad(\sF(\bigvee S)) = \rad(\bigcup_{X \in S} \sF(X)) = \bigvee_{X\in S} \rad(\sF(X)) \quad \text{and} \\
\rad(\sF(\rad(X) \wedge \rad(Y))) = \rad(\sF(X \otimes Y)) = \rad(\sF(X)) \wedge \rad(\sF(Y))
\end{gather*}
for any set of objects $S$ in $\cat{A}$ and objects $X$ and $Y$ in $\cat{A}$. Hence
\begin{equation*}
\rad(\sF(-)) \colon \Rad(\cat{A}) \to \Rad(\cat{B})
\end{equation*}
is a morphism of coherent frames. Applying Stone and Hochster duality we obtain a continuous map
\begin{equation*}
\balmer{\cat{B}} \to \balmer{\cat{A}} \quad \text{given by} \quad \cat{Q} \mapsto \sF^{-1}(\cat{Q})\,.
\end{equation*}

\begin{lemma} \label{surj_functor_inj_balmer}
Let $(\sF,\beta) \colon (\cat{A},\otimes_\cat{A},\unit_\cat{A}) \to (\cat{B},\otimes_\cat{B},\unit_\cat{B})$ be an essentially surjective tensor extriangulated functor. Then the induced map $\balmer{\cat{B}} \to \balmer{\cat{A}}$ is injective. If for every extriangle $X \to Y \to Z \dashrightarrow$ in $\cat{B}$ there exists an extriangle $X' \to Y' \to Z' \dashrightarrow$ with $\sF(X') \cong X$, $\sF(Y') \cong Y$ and $\sF(Z') \cong Z$, then the injective map is an embedding.
\end{lemma}
\begin{proof}
The injectivity is clear from the definition of the induced map. 

For the second claim it is enough to show $\cat{I} \subseteq \sF^{-1}(\cat{Q})$ if and only if $\rad(\sF(\cat{I})) \subseteq \cat{Q}$ for any radical thick tensor ideal $\cat{I} \subseteq \cat{A}$ and any $\cat{Q} \in \balmer{\cat{B}}$. By assumption we have $\rad(\sF(\cat{I})) = \sF(\cat{I})$. Hence the forward direction holds. The backward direction is clear.
\end{proof}

The Balmer spectrum has been computed for many tensor triangulated categories, like the perfect complexes of a commutative ring \cite{Hopkins:1987,Neeman:1992b,Thomason:1997} and the stable module category of a group algebra \cite{Benson/Carlson/Rickard:1997}. In the following we describe the Balmer spectrum of some of the examples from \cref{sec:example_tec}. 

%%%%%%%%%%%%%%%%%%%%%%%%%%%%%%%%%%%%%%%%%%%%%%%%%%%%%%%%%%%%
\subsection{Projective modules} \label{teg_projective_modules}

Let $R$ be a commutative ring with local units. In \cref{example-local-units-flat} we saw that the category of flat modules $\Flat{R}$ is a tensor extriangulated category. This category need not be essentially small. However, the subcategory of finitely presented flat modules is essentially small and so we may consider its Balmer spectrum. In fact, a module is finitely presented and flat if and only if it is finitely presented and projective. In the following we compute the Balmer spectrum of $\proj{R}$ for some classes of commutative rings $R$.

Note that the exact structure on $\proj{R}$ is the split exact structure, so the radical thick tensor ideals are precisely the radical tensor ideals that are closed under finite coproducts and direct summands.

\begin{lemma} \label{proj_BSpec_point}
Let $R$ be a commutative ring that satisfies one of the following conditions: 
\begin{enumerate}
\item \label{proj_BSpec_point:polynomial} $R=S[x_{1},\dots,x_{n}]$, the polynomial ring over a  principal ideal domain $S$; 
\item \label{proj_BSpec_point:local} $R$ is a local ring; or
\item \label{proj_BSpec_point:Dedekind} $R$ is a Dedekind domain.
\end{enumerate}
Then $\balmer{\proj{R}} = \{\rad(0)\}$; that is the Balmer spectrum consists of one point. 
\end{lemma}
\begin{proof}
In the cases \cref{proj_BSpec_point:polynomial,proj_BSpec_point:local} every finitely generated projective $R$-module is free; for \cref{proj_BSpec_point:polynomial} this follows from the Quillen--Suslin theorem and for \cref{proj_BSpec_point:local} from Kaplansky's theorem; see for example \cite[Chapter~XXI, Theorem~3.7]{Lang:2002} and \cite[Theorem~2]{Kaplansky:1958}, respectively. Hence there are exactly two radical thick tensor ideals: $\rad(0)$ and $\proj{R}$, and $\rad(0)$ is the only prime thick tensor ideal.

In the case \cref{proj_BSpec_point:Dedekind} the finitely projective $R$-modules are precisely the torsion-free modules. In particular, any ideal is projective. Let $I$ and $J$ be non-zero ideals of $R$. By \cite[Exercise~19.5]{Eisenbud:1995}, we obtain $J \oplus I^2 \cong JI \oplus I$. Since $I$ is projective, we have $JI \cong J \otimes I$. This means $J \in \thick(I)$. Furthermore, any finitely generated projective $R$-module is of the form $R^{n} \oplus I$ for some ideal $I$; see \cite[Exercise~19.6]{Eisenbud:1995}. So $\proj{R}$ has two radical thick tensor ideals, namely $\rad(0)$ and $\proj{R}$, and $\rad(0)$ is the only prime thick tensor ideal.
\end{proof}

\begin{lemma} \label{proj_BSpec_product}
Let $R = R_1 \times \ldots \times R_n$ be a product of commutative rings. Then
\begin{equation*}
\balmer{\proj{R}} = \balmer{\proj{R_1}} \times \ldots \times \balmer{\proj{R_n}}
\end{equation*}
with the product topology. 
\end{lemma}
\begin{proof}
By induction it is enough to show the claim for $n=2$. Let $R = S \times T$. As an $R$-module $R \cong S \oplus T$ and every $R$-module decomposes as $M \cong (S \otimes_R M) \oplus (T \otimes_R M)$. In particular, $M$ is a finitely presented projective $R$-module if and only if $S \otimes_R M$ and $T \otimes_R M$ are finitely presented projective over $S$ and $T$, respectively. The map
\begin{equation*}
\begin{gathered}
\Rad(\proj{R}) \to \Rad(\proj{S}) \times \Rad(\proj{T}) \,, \\
\cat{I} \mapsto \set{S \otimes_R M | M \in \cat{I}} \times \set{T \otimes_R M | M \in \cat{I}}
\end{gathered}
\end{equation*}
has an inverse given by
\begin{equation*}
\rad(\cat{I}\cup \cat{J}) \mapsfrom \cat{I} \times \cat{J}
\end{equation*}
where we view the objects in $\cat{I}$ and $\cat{J}$ as $R$-modules using restriction. This bijection is an isomorphism of coherent frames and hence it yields the desired isomorphism of the Balmer spectra.
\end{proof}

Note that a commutative noetherian ring is hereditary if and only if it is regular of Krull dimension 1; see for example \cite[Corollary~3]{Endo:1961}. For such rings \cref{proj_BSpec_reg_Kdim1} describes the Balmer spectrum. 

\begin{corollary} \label{proj_BSpec_reg_Kdim1}
If $R$ is a commutative hereditary noetherian ring, then  the Balmer spectrum $\balmer{\proj{R}}$ is the finite set of connected components of the Zariski spectrum $\ZSpec{R}$ equipped with the discrete topology, and $\vert \balmer{\proj{R}}\vert$ is equal to the number of indecomposable factors of $R$.
\end{corollary}
\begin{proof}
Note that $R$ cannot contain an infinite set of primitive idempotents since it is noetherian; see for example \cite[Proposition 22.2]{Lam:2001}. By \cite[Theorem~4.4]{Bergman:1971}, see also \cite[p.~44,~Corollary~1]{Chatters-Jondrup:1983}, the ring $R$ is a finite product of Dedekind domains and fields. Hence the claim follows from \cref{proj_BSpec_point,proj_BSpec_product}. 
\end{proof}

%%%%%%%%%%%%%%%%%%%%%%%%%%%%%%%%%%%%%%%%%%%%%%%%%%%%%%%%%%%%%
\subsection{Stabilisation of extriangulated categories} \label{teg_stabilisation}

We discussed the stabilisation of a tensor extriangulated category in \cref{subsec:stabilisation_ttc}. In the stabilisation $\underline{\cat{A}}_\cat{I}$ the objects in $\cat{I}$ are isomorphic to $0$. Hence there are fewer radical thick tensor ideals in $\underline{\cat{A}}_\cat{I}$ than in $\cat{A}$. In fact, the identity induces a tensor extriangulated functor $\cat{A} \to \underline{\cat{A}}_\cat{I}$, and as a direct consequence of \cref{surj_functor_inj_balmer} we obtain:

\begin{lemma}
Let $(\cat{A},\otimes,\unit)$ be a tensor extriangulated category and $\cat{I}$ a tensor ideal consisting of objects that are $\BE$-projective and $\BE$-injective. Then
\begin{equation*}
\balmer{\underline{\cat{A}}_\cat{I}} = \set{\cat{P} \in \balmer{\cat{A}} | \cat{P} \supseteq \cat{I}}
\end{equation*}
with the subspace topology. \qed
\end{lemma}

\begin{example} \label{teg_hopf}
Let $H$ be a cocommutative Hopf algebra over a field $k$. We assume that $H$ is finite-dimensional over $k$. Then the category finite-dimensional $H$-modules is essentially small, and by \cref{example_Hopf} it becomes a tensor extriangulated category with $(\otimes_k,k)$. Further, $\mod{H}$ is a Frobenius exact category, and there exists exactly one non-trivial thick tensor ideal consisting of objects that are projective and injective. We denote by $\smod{H}$ the stabilisation of $\mod{H}$ with respect to this tensor ideal. Hence under the induced morphism of coherent frames
\begin{equation*}
\Rad(\mod{H}) \to \Rad(\smod{H})
\end{equation*}
every ideal has precisely one pre-image except the zero ideal in $\smod{H}$, which has exactly two pre-images: the zero ideal and the ideal of all projective objects.

Let $G$ be a finite group. By \cite{Benson/Carlson/Rickard:1997}, 
\begin{equation*}
\balmer{\smod{kG}} = \Proj{\ch{G,k}}\,,
\end{equation*}
where $\ch{G,k}$ is the group cohomology of $G$ and $\Proj{\ch{G,k}}$ the set of homogenous primes without the irrelevant ideal with the Zariski topology. Then
\begin{equation*}
\balmer{\mod{kG}} = \Spec{\ch{G,k}}
\end{equation*}
where the latter is the set of all homogenous prime ideals; also see \cite[Example~17]{Krause:2024}. 
\end{example}

\begin{example}
Let $(\cat{A},\otimes,\unit)$ be an additive monoidal category and $\ccat{\cat{A}}$ the category of chain complexes. The degree-wise split exact sequences endow $\ccat{\cat{A}}$ with an exact structure. It is straightforward to check that $\ccat{\cat{A}}$ is a tensor extriangulated category with the tensor product induced by $\otimes$ and the unit the complex with $\unit$ concentrated in degree zero. Moreover, the exact structure on $\ccat{\cat{A}}$ is Frobenius and the projective objects are precisely the complexes of the form $\cone(\id_X)$ for any complex $X$; these are precisely the contractible complexes. 

We endow $\cat{A}$ with the split exact structure. Then there is a bijection
\begin{gather*}
\Rad(\cat{A}) \to \set{\cat{J} \in \Rad(\ccat{\cat{A}}) | \cat{J} \subseteq \Proj{\ccat{\cat{A}}}} \\
\text{given by} \quad \cat{I} \mapsto \rad(\set{\cone(\id_M) | M \in \cat{I}})\,.
\end{gather*}
Hence there may exist indeterminate extriangulated stabilisations between the category of chain complexes $\ccat{\cat{A}}$ and its homotopy category 
\begin{equation*}
\kcat{\cat{A}} = \ccat{\cat{A}}/\Proj{\ccat{\cat{A}}}\,;
\end{equation*}
compare with  \cref{proj_BSpec_product,proj_BSpec_reg_Kdim1}.
\end{example}

\bibliographystyle{amsalpha}
\bibliography{References}

\end{document}